\documentclass[11pt]{amsart}
\usepackage{amssymb}
\usepackage{amsmath}
\usepackage{amsthm}
\usepackage{amsfonts}
\usepackage{bbm}
\usepackage{enumitem}
\usepackage{epsf,epsfig,amsfonts,graphicx}
\usepackage{color}
\usepackage[T1]{fontenc}
\newtheorem{theorem}{Theorem}

\newtheorem{remark}{Remark}
\newtheorem{proposition}{Proposition}
\newtheorem{corollary*}{Corollary}

\newtheorem{definition}{Definition}
\newcommand{\A}{B}

\newcommand{\Aas}{A(s)}
\newcommand{\al}{\alpha}

\newcommand{\as}{A(s)}
\newcommand{\asminus}{\left(A(s)-
\mathbbm{1}/2\right)^{-1}}
\newcommand{\asa}{\left(A(s)-
\mathbbm{1}/2\right)}

\newcommand{\bas}{b(s)}
\newcommand{\be}{\begin{equation}}
\newcommand{\ee}{\end{equation}}
\newcommand{\ben}{\begin{equation*}}
\newcommand{\bs}{b(s)}
\newcommand{\een}{\end{equation*}}
\newcommand{\cas}{c_{s}}

\newcommand{\co}{{\rm co}}
\newcommand{\coA}{{\rm co}_B}

\newcommand{\coanew}{{\rm cos}_b}
\newcommand{\complexnumbers}{\mathbb{C}}

\newcommand{\cvw}{c[v,w]}
\newcommand{\dotcas}{{\dot{c}_{s}}}

\newcommand{\grad}{{\rm grad}}
\newcommand{\halb}{\frac{1}{2}}
\newcommand{\Identity}{\mathbbm{1}}

\newcommand{\mnr}{M(n;\mathbb{R})}
\newcommand{\mtwor}{M(2;\mathbb{R})}

\newcommand{\po}{\mathcal{P}}
\newcommand{\poN}{\mathcal{P}_N}
\newcommand{\R}{\mathbb{R}}
\newcommand{\rn}{\mathbb{R}^n}
\newcommand{\siA}{{\rm si}_B}
\newcommand{\si}{{\rm si}}

\newcommand{\sianew}{{\rm sin}_b}

\newcommand{\tcas}{\widetilde{c}_{s}}
\newcommand{\T}{\mathcal{T}}
\newcommand{\Ta}{T}

\newcommand{\viertel}{\frac{1}{4}}
\newcommand{\weg}[1]{}
\newcommand{\z}{\mathbb{Z}}

\title{Solitons of discrete curve shortening}
\author{Christine Rademacher}
\address{Technische Hochschule N\"urnberg Georg Simon Ohm,
Fakult\"at Angewandte Mathematik, Physik und Allgemeinwissenschaften,
Postfach 210320, 90121 N\"urnberg, Germany}
\email{christine.rademacher@th-nuernberg.de}
\author{Hans-Bert Rademacher}
\address{Universit\"at Leipzig, Mathematisches Institut, 04081 Leipzig, Germany}
\email{rademacher@math.uni-leipzig.de}
\date{2015-08-27, revised: 2016-06-16}
\thanks{We are very grateful for the comments and suggestions
of the referees}
\begin{document}

\begin{abstract}
For a polygon $x=(x_j)_{j\in \z}$ in $\R^n$ we consider
the polygon  $(T(x))_j=\left\{x_{j-1}+2x_j+x_{j+1}\right\}/4\,.$
This transformation is
 obtained by applying the midpoints
polygon construction twice. %\marginpar{Omit {\em using...}}
For a closed polygon or a polygon with finite vertices 
this is a curve shortening process.
We call a polygon $x$ a {\em soliton} of the 
transformation $T$ 
if the polygon $T(x)$ is an affine image of $x.$ 
We describe a large class of solitons for the transformation $T$ 
by considering smooth curves
$c$ which are solutions of the differential
equation $\ddot{c}(t)=Bc(t)+d$ for a
real matrix
$B$ and a vector
$d.$ 
The solutions of this 
differential equation can be written in terms of 
power series in the matrix $B.$
For a solution $c$
and for any $s>0,a\in \R$
the  polygon $x(a,s)=(x_j(a,s)_j)_{j\in \z};x_j(a,s)=c(a+sj)$
is a soliton of $T.$
For example we obtain solitons lying on
spiral curves which under the
transformation $T$ rotate and shrink.
%%%%%%%%%%%%%%%%%%%%%%%%%%%%%%%%%%%%%%%%%%%%%%%%%%
%%%%%%%%%%%%%%%%%%%%%%%%%%%%%%%%%%%%%%%%%%%%%%%%%%%
%%%%%%%
%The transformation $T$ can also be seen as a discretization of
%the negative gradient flow of the energy
%functional on the space of 
%closed polygons.
\end{abstract}
\keywords{discrete curve shortening, polygon, 
affine mappings, soliton, midpoints polygon, 
linear system of ordinary differential equations}
\subjclass[2010]{53C44 (34A05 34A30)}

\maketitle

\section{Introduction}

For an
infinite polygon $(x_j)_{j\in \z}$ given by the vertices 
$x_j$  in the vector space $\R^n$ 
the {\em midpoints
polygon} is defined by $M(x)$ where $M(x)_j:=(x_j+x_{j+1})/2; j\in \z.$ 
If the polygon is closed or rather periodic, i.e. if for some $N\,:$
$x_{j+N}=x_j$ for all $j\in \z$ then this midpoint mapping $M$
defines a {\em curve shortening process,}
i.e. the polygon $M(x)$ is shorter 
than the polygon $x$ unless it is a single point.
If we iterate the process for a closed polygon $(x_j)_{j=1,\ldots,N}$ 
it converges to the barycenter
$(x_1+\ldots+x_N)/N.$
This elementary construction was already used by Darboux in 1878. 
He also showed in~\cite{Darboux1878} that 
in the plane in the general case the sequence $M^k(x)/\cos(\pi/N)^k$
converges for $k \to \infty$  to an ellipse. 
These results 
for polygons in Euclidean space 
were later rediscovered and extended by several authors,
for example 
Kasner~\cite{Kasner1903},
Schoeneberg~\cite{Schoenberg1950}, and
Berlekamp et al.~\cite{Berlekamp1965}.
Bruckstein \& Shaked~\cite{Bruckstein1997} discuss
the relation with iterative 
smoothing procedures in shape analysis and recognition.
Nowadays applications of 
discrete curve shortening are  discussed in several papers.
For example Smith et al.~\cite{Smith2007} 
present its connection with the 
rendezvous problem for mobile autonomous 
robots.

We modify the curve shortening process $M$
for an infinite polygon (i.e. not necessarily closed)
in Euclidean space as follows:
Instead of the midpoint mapping $M$ we apply the midpoint 
mapping twice
and use an index shift, i.e. we define the polygon 
$T(x)=(T(x)_j)_{j\in \z}$ by the Equation
\begin{equation}
T(x)_j=\frac{1}{4}\left\{x_{j-1}+2x_j+x_{j+1}\right\}\,,
\end{equation}
i.e. $T(x)_j=\left(M^2(x)\right)_{j-1}.$ 
Introducing the index shift has the following 
advantage: 
Since
\begin{equation}
\label{eq:semi1}
T(x)_j -x_j =\frac{1}{4}\left\{x_{j-1}-2x_j+x_{j+1}\right\}
\end{equation}
we can view this
process $T$ as a discrete approximation of the 
{\em semidiscrete flow} 
$s\mapsto x_j(s)$ defined by
\begin{equation}
\label{eq:semi2}
\frac{d x_j(s)}{ds}=
x_{j-1}(s)-2 x_j(s)+x_{j+1}(s)\,,\,x_j(0)=x_j
\end{equation}
on the space of polygons. 
On the space of closed
polygons this flow is the negative gradient flow of 
%%%%%%%%%%%%%%%%%%%%%%%%%%%%%%%%%%%%%%%%%%%%%%%%%%%%%%%%%
%%%%%%%%%%%%%%%%%%Aenderung 2
%%%%%%%%%%%%%%%%%%%%%%%%%%%%%%%%%%%%%%%%%%%%%%%%%%%%%%%%%%
%the energy
the functional $F_2,$ 
cf. Section~\ref{sec:semidiscrete}.
Semidiscrete flows are 
discussed for example by 
Chow\& Glickenstein~\cite{Chow2007}.

The polygon $T(x)$ is formed by the midpoints $T(x)_j$ of the 
medians through $x_j$ of the triangle formed by
$x_{j-1},x_j,x_{j+1}.$
Linear
polygons $x_j=j u  +v\,;\, u,v \in \R^n$ are the fixed points of $T.$

Since the process $T$ is affinely invariant it is natural  to consider
polygons which are mapped under $T$ onto an affine image of themselves.
We also use the
term {\em soliton} for these polygons
in analogy to the case of smooth curves
or manifolds 
which are mapped under the mean curvature flow
onto the same curve 
or rather
manifold up to 
an isometry,
cf.
Hungerb{\"u}hler \& Smoczyk~\cite{Hungerbuehler2000},
Hungerb{\"u}hler \& Roost~\cite{Hungerbuehler2011}
and 
Altschuler~\cite{Altschuler2013}.
We call a polygon $x=\left(x_j\right)_j$ a {\em soliton}
of the process $T$ if there exists an affine
mapping $x\in \R^n\mapsto Ax+b \in \R^n$ for 
a matrix $A$ and a vector $b$ such that
\begin{equation}
\label{eq:polysoliton}
T(x)_j
=\frac{1}{4}\left\{x_{j-1}+2x_j+x_{j+1}\right\}
= A x_j+b
\end{equation}
for all $j \in \z.$
The main idea of this paper is to consider not only polygons, 
but also smooth curves $c$ which are affinely invariant under
an analogous process on curves:
We adapt
the process $T$ to curves
in the following sense:\\
We associate the one-parameter family $c_s,
s\in \R$
to the smooth curve $c:$
\begin{equation}
\label{eq:oneparameter}
c_s(t):=
%\T\left(c(t-s),c(t),c(t+s)\right)=
\frac{1}{4}\left\{c(t-s)+2 c(t)+c(t+s)\right\}\,.
\end{equation}
% % % % % % % % % % % % % % % % % % % % % % % % %
%Using the affine map
%$x,y,z \in \R^n\mapsto %\mathcal{T}(x,y,z)=\left\{x+2y+z\right\}/4
%\in \R^n,$ we can also write:
%$$c_s(t)=
%\mathcal{T}\left(c(t-s),c(t),c(t+s)\right).$$
% % % % % % % % % % % % % % % % % % % % % % %
For $s>0,a\in \R$ we define the polygon
$x(a,s)=(x_j(a,s)_j)_{j\in \z};x_j(a,s)=c(a+sj).$
Then the smooth curve $c$ defines solitons 
of the form
$x(a,s)$ for the curve shortening
process $T$ if the following holds:
For some $\epsilon >0$ and any $s\in (0,\epsilon)$ 
there is a one-parameter family of affine maps
$x \in \R^n\mapsto \as x +\bs \in \R^n$
such that for any $s\in (0,\epsilon)$
and all $t \in \R:$
{%
\renewcommand{\theequation}{$*$}%
  \addtocounter{equation}{-1}%
\begin{equation}
%\ref{\ast}
\label{eq:hauptgleichung}
c_s(t)=
%\T\left(c(t-s),c(t),c(t+s)\right)=
\frac{1}{4}\left\{c(t-s)+2 c(t)+c(t+s)\right\}
=\as c(t)+\bs
\,.
\end{equation}
}
\begin{figure}[t]
  \centering
    \begin{minipage}[b]{6.2 cm}
    \includegraphics[width=0.9\textwidth]{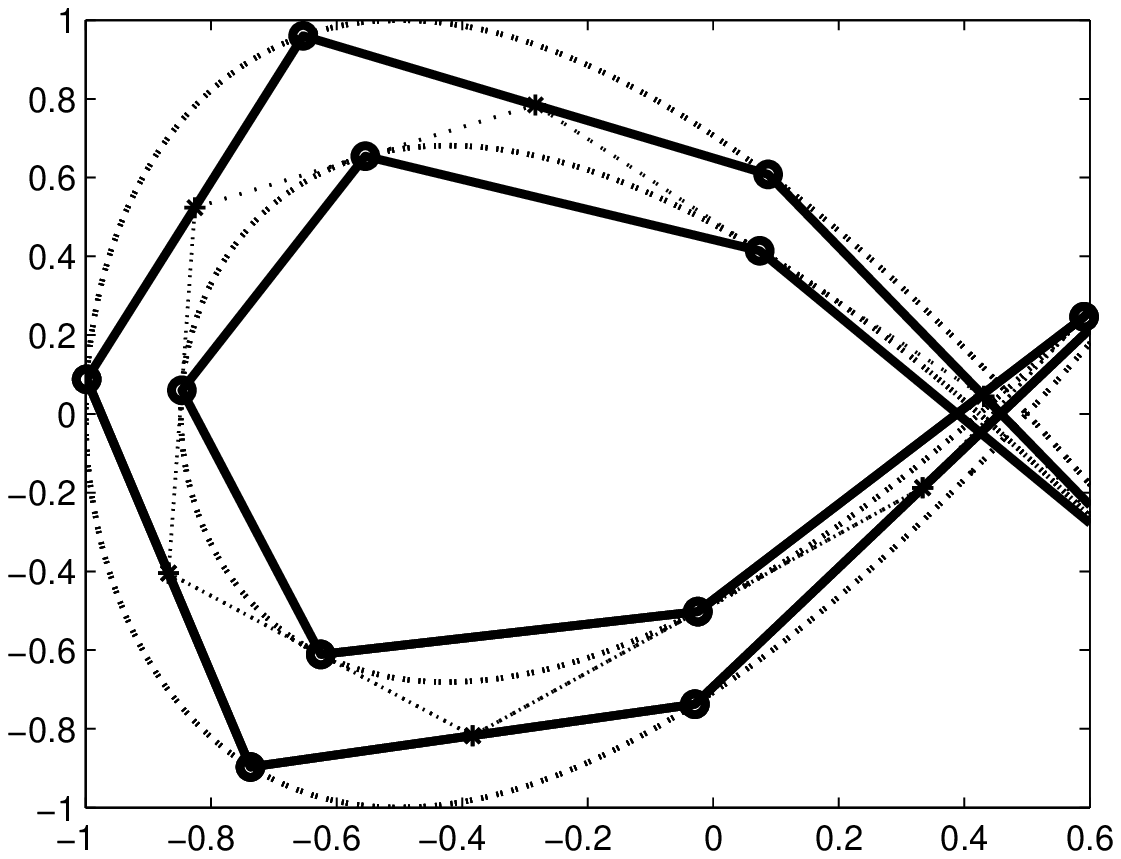} 
    \caption{polygon as soliton of $T$}
    \label{fig:polygon1A}
  \end{minipage}
  \begin{minipage}[b]{6.2 cm}
      \includegraphics[width=0.95\textwidth]{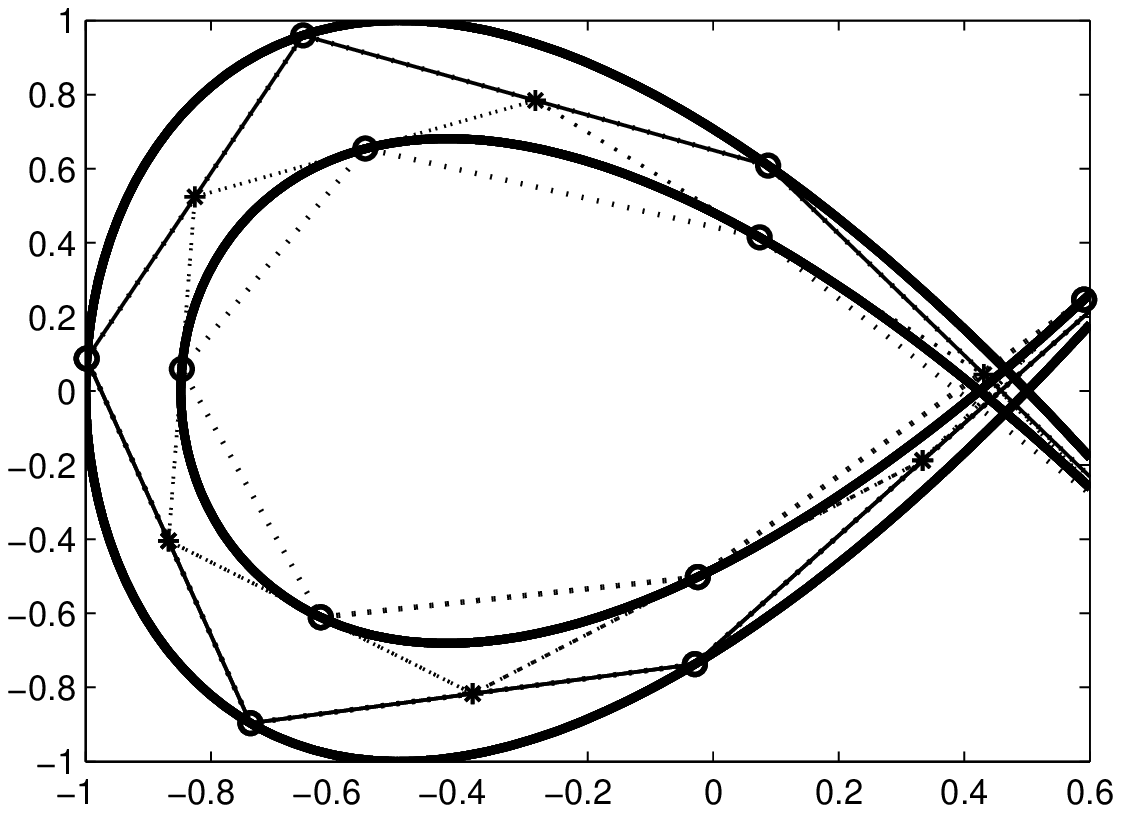} 
      \caption{smooth curve as soliton of $T$}
      \label{fig:polygon2}
    \end{minipage} 
    \end{figure}
% % % % % % % % % % % % % % % % % % % % % %
In Figure~\ref{fig:polygon1A} we show a polygon
$x=(x_j)$ which is a soliton of the process $T.$
Its vertices $x_j=c\left(0.4 \cdot j\right), j\in \z$
lie on the smooth curve
$c(t)=\left(\cos(2t),\cos(3t)\right),$ which 
is also a soliton of the process $T,$
see Figure~\ref{fig:polygon2}.
Denote by $A$ the diagonal matrix with
entries $(1+\cos(0.8))/2$ and $(1+\cos(1.2))/2.$
Then in Figure~\ref{fig:polygon1A} the polyon
$x$ and its image $T(x)$ are shown, which satisfy
Equation~\eqref{eq:polysoliton}. 
In Figure~\ref{fig:polygon2} the smooth curves
$c$ and $c_{0.4}$ are shown, which satisfy
Equation~\eqref{eq:hauptgleichung}
for $s=0.4, A(0.4)=A,b(0.4)=0.$
Hence the process $T$ in this case corresponds to 
a {\em scaling.} This example belongs to Case (1a)
in Section~\ref{sec:planar}. 
Here $A(0.4)$ is given by Equation~\eqref{eq:A(s)}
where $B$ is the diagonal matrix with entries
$b_1=-4$ and $b_2=-9.$

\medskip

% % % % % % % % % % % % % % % % % % % % % % %
It is the main result of this paper that 
solutions of Equation~\eqref{eq:hauptgleichung}, i.e.
smooth curves $c$
defining solitons for the curve shortening process $T,$ 
can be characterized as solutions of
the inhomogeneous linear differential equation 
\begin{equation}
\label{eq:DGL} \ddot{c}= \A c(t)+d
\end{equation}
of second order with constant coefficients with
$B=2\;A''(0),d=2\; b''(0),$ cf. Theorem~\ref{thm:characterization}
and Theorem~\ref{thm:main}. 
The solitons $c$ and the
maps $\as$ 
or rather
$\bs$ can be described in terms 
of the power series
\be
\coA(t):=\sum_{k=0}^{\infty}\frac{t^{2k}}{(2k)!}B^k
%\Identity+\frac{t^2}{2!}A+\frac{t^4}{4!}A^2+\ldots
\enspace,
\enspace
\siA(t):=\sum _{k=0}^{\infty}\frac{t^{2k+1}}{(2k+1)!}B^k\,,
%\Identityentity+\frac{t^2}{2!}A+\frac{t^4}{4!}A^2+\ldots
\ee
cf. Proposition~\ref{pro:ode} and
Proposition~\ref{pro:odeb} in Section~\ref{sec:linear} .\\
We will show, that for any real matrix $B$  solutions of Equation~\eqref{eq:DGL}
are solitons of the curve shortening process
and the  matrix $\as$ of 
Equation~\eqref{eq:hauptgleichung} is given by
\begin{equation}
\label{eq:A(s)}
\as=
\frac{1}{2}\left\{\Identity +\coA(s)
\right\}
\, .
\end{equation}
The vectors $\bs$ depend on the structure 
of the matrix $B$ and vanish for closed
solitons.\\
The one-parameter family $s \mapsto c_s(t)$ 
defined by Equation~\eqref{eq:oneparameter} associated
to a soliton $c$ can be used to
define an affinely invariant solution of the 
{\em wave equation,} cf. Remark~\ref{rem:wave}.
The relation of polygonal curve shortening and the
curve shortening flow for smooth curves is discussed
in many papers, cf.~\cite[Sec.3]{Sapiro1995},
a reference for curve shortening flows is the book~\cite{Chou2001}.
Self-similar solutions 
of the Euclidean curve shortening flow
in the plane,
i.e. the mean curvature flow for 
curves in Euclidean space,
are discussed by 
Halldorsson~\cite{Halldorsson2012}, see also 
Hungerb{\"u}hler \& Smoczyk~\cite{Hungerbuehler2000} and
Altschuler~\cite{Altschuler2013}.
These questions lead to systems of non-linear ordinary differential
equations.

In Section~\ref{sec:semidiscrete} we show that the
solitons of the curve shortening $T$ are also solitons of the
{\em semidiscrete flow} 
defined in Equation~\eqref{eq:semi2}.
In Section~\ref{sec:planar} we discuss the planar case $n=2$
in detail and present examples.
For example the {\em parabola} is a soliton for which
the curve shortening leads to a {\em translation.}
We also obtain 
various spirals as solitons which rotate and shrink
and closed curves of Lissajous type, which scale under the mapping $T.$
The results of Section~\ref{sec:planar} should
be compared with the {\em zoo of solitons}
obtained for the Euclidean curve shortening flow 
by
Halldorsson~\cite{Halldorsson2012}, see also 
Hungerb{\"u}hler \& Smoczyk~\cite{Hungerbuehler2000} and
Altschuler~\cite{Altschuler2013}.
%%%%%%%%%%%%%%%%%%%%%%%%%%%%%%%%%%%%%%%%%%%%%%%%%%%%%%%%%%%%%%
%%%%%%%%%%%%%%%%%%%%%%%%%%%%%%%%%%%%%%%%%%%%%%%%%%%%%%%%%%%%%%%%%%%%%%
%%%%%%%%%%%%%%%%%%%%%%%%%%%%%%%%%%%%%%%%%%%%%%%%%%%%%%%%%%%%%%%%%%%%%%%
\section{System of linear differential equations of second order}
%%%%%%%%%%%%%%%%%%%%%%%%%%%%%%%%%%%%%%%%%%%%%%%%%%%%%%%%%%%%%%
%%%%%%%%%%%%%%%%%%%%%%%%%%%%%%%%%%%%%%%%%%%%%%%%%%%%%%%%%%%%%%%%%%%%%%
%%%%%%%%%%%%%%%%%%%%%%%%%%%%%%%%%%%%%%%%%%%%%%%%%%%%%%%%%%%%%%%%%%%%%%%
\label{sec:linear}
%%%%%%%%%%%%%%%%%%%%%%%%%%%%%%%%%%%%%%%%%%%%%%%%%%%%%%%%%%%%%%
%%%%%%%%%%%%%%%%%%%%%%%%%%%%%%%%%%%%%%%%%%%%%%%%%%%%%%%%%%%%%%%%%%%%%%
%%%%%%%%%%%%%%%%%%%%%%%%%%%%%%%%%%%%%%%%%%%%%%%%%%%%%%%%%%%%%%%%%%%%%%%
Let $B\in \mnr$ be a $n \times n$ matrix 
with real entries. We denote by $\Identity$ the identity matrix and
use the following notation:\\
Using the {\em matrix exponential} 
$\exp(B):=\sum_{k=0}^{\infty}B^k/k!$
one can also define the matrix functions 
$\cosh(B),\cos(B),\sin(B),\sinh(B)$ as power series in
$B.$
For commuting matrices $B,C$ (i.e.
$BC=CB$) we have
$$
\exp(B+C)=\exp(B)\exp(C)=\exp(C)\exp(B)\; ,
$$
cf. \cite[ch.6.1]{Kuehnel2011}.
%\,,
%\end{equation*}

%%%%%%%%%%%%%%%%%%%%%%%%%%%%%%%%%%%%%%%%%%%%%%%%%%%%%%%%%%
\begin{definition}
\label{def:cosA}
For a matrix $B \in M(n,\R)$ we define
the following mappings $\coA, \siA: \R \rightarrow M\left(n,\R\right):$
\begin{eqnarray*}
\coA(t)&:=&\sum_{k=0}^{\infty}\frac{t^{2k}}{(2k)!}B^k=
\Identity+\frac{t^2}{2!}B+\frac{t^4}{4!}B^2+\ldots\\
\siA(t)&:=&\sum_{k=0}^{\infty}\frac{t^{2k+1}}{(2k+1)!}B^k
=t \Identity+\frac{t^3}{3!}B+\frac{t^5}{5!}B^2+\ldots
\end{eqnarray*}
%\hfill\qedsymbol
\end{definition}
%%%%%%%%%%%%%%%%%%%%%%%%%%%%%%%%%%%%%%%%%%%%%%%%%%%%
These power series are obviously convergent and we obtain
the following
%%%%%%%%%%%%%%%%%%%%%%%%%%%%%%%%%%%%%%%%%%%%%%%%%%%%%%%
\begin{proposition}[{\sc Properties of  $\coA(t),\siA(t)$}]
\label{pro:derivative-coA}
For a matrix $B$ the above defined
mappings $\coA(t),\siA(t)$ 
satisfy the following 
(differential) equations:
\begin{eqnarray*}
\frac{d}{dt} \coA (t)=\coA'(t)=B \cdot\siA (t)\,;\,
\siA' (t)=\coA (t)\\
\coA''(t)=B\cdot\coA(t)\,;\,
\siA''(t)=B\cdot\siA(t)\\
\coA(t)^2-B \cdot\siA(t)^2= \Identity\,.
\end{eqnarray*}
\end{proposition}
%%%%%%%%%%%%%%%%%%%%%%%%%%%%%%%%%%%%%%%%%%%%%%%%%%%%%%%%%
Given a real number $b \in \R$ we denote by
$\coanew, \sianew:
\R \rightarrow \R$ the unique solutions of the
differential equation
\protect
$f''=bf$ 
with
$\coanew(0)=1,\coanew'(0)=0, 
\sianew(0)=0, \sianew'(0)=1.$ Hence 
\begin{equation}
\label{eq:coa}
\coanew(t)=
\left\{
\begin{array}{ccc}
\cosh\left(\sqrt{b}t\right)&;& b>0\\
1 &;& b=0\\
\cos\left(\sqrt{-b}t\right)&;& b<0
\end{array}
\right.
\,;\,
\sianew(t)=
\left\{
\begin{array}{ccc}
\sinh\left(\sqrt{b}t\right)/\sqrt{b}&;& b>0\\
t &;& b=0\\
\sin\left(\sqrt{-b}t\right)/\sqrt{-b}&;& b<0
\end{array}
\right.
\end{equation}
Then we obtain 
for a real number $b$ and the matrix
$B=b \Identity:$
\begin{equation*}
\coA(t)=\co_{b\Identity}(t)=\coanew(t) \Identity\,;\, 
\siA(t)=\si_{b\Identity}(t)=\sianew(t) \Identity\,.
\end{equation*}
If $B=C^2, C\in M\left(n,\complexnumbers\right)$ then
$\coA(t)=\cosh(Ct), C \cdot\siA(t)=\sinh(tC).$ For 
$B=-C^2, C\in M\left(n,\complexnumbers\right)$ we obtain
$\coA(t)=\cos(Ct), C \cdot \siA(t)=\sin(Ct).$
%%%%%%%%%%%%%%%%%%%%%%%%%%%%%%%%%%%%%%%%%%%%%%%%%%%%%%%%%%%%%%%%
\begin{proposition}[{\sc Homogeneous differential equation}]
\label{pro:ode}
For a matrix $B \in \mnr$ and vectors
$v,w\in \R^n$
the
linear system of ordinary differential equations
(with constant coefficients) of second order:
\begin{equation}
\label{eq:linearode}
\ddot{c}(t)=B c(t)
\end{equation}
with initial values
$c(0)=v, \dot{c}(0)=w$
has the unique solution
\begin{equation}
\label{eq:ctvw}
c(t)=\cvw(t):=\coA(t)(v)+\siA(t)(w)\,.
\end{equation}
Then the one-parameter family of curves $s\mapsto \cas$ 
for $s \in \R$
defined by 
Equation~\eqref{eq:oneparameter}
satisfies
\begin{equation}
\label{eq:csdef}
\cas(t)= \as \;
c(t)
\,,
\end{equation} 
with
$$\as=
\frac{1}{2}\left\{\Identity+\coA(s) \right\}$$
i.e. $c_s$ 
is a linear image of $c$
for any $s$
and Equations~\eqref{eq:hauptgleichung} and Equations~\eqref{eq:A(s)} hold
with $\bs=0$ .
\end{proposition}
%%%%%%%%%%%%%%%%%%%%%%%%%%%%%%%%%%%%%%%%%%%%%%%%
\begin{proof}
From Proposition~\ref{pro:derivative-coA} 
it follows that $c(t)=\cvw(t)$ is a solution of
Equation~\eqref{eq:linearode}
with $c(0)=v, \ddot{c}(0)=w.$
Equation~\eqref{eq:oneparameter} 
implies that $\cas$ is the unique solution
of Equation~\eqref{eq:linearode} with initial conditions:
\begin{eqnarray*}
\cas(0)&=& \frac{1}{4}
\left\{ c(-s)+ 2 c(0)+ c(s)
\right\}=
\frac{1}{2}\left\{
\Identity+\coA (s)\right\}(v) 
\\
\dotcas(0)&=&
\frac{1}{4}
\left\{
\dot{c}(-s)+2\dot{c}(0)+  \dot{c}(s)
\right\}
 =
\frac{1}{2}
\left\{\Identity+\coA (s)
\right\}(w) \,.
\end{eqnarray*}
Using Equation~\eqref{eq:ctvw}
we can write
\begin{eqnarray*}
\cas(t)&=& \frac{1}{2}\,
\coA(t)
\left\{\Identity+\coA(s)\right\}(v)+
\,
\siA(t)
\left\{\Identity+\coA(s)\right\}(w)
\\
&=&
\frac{1}{2}
\,\left\{\Identity+ \coA(s)\right\}\,
\left\{\coA(t)(v)+\siA(t)(w)\right\}
=
\frac{1}{2}
\,\left\{\Identity+\coA(s)\right\}\,
\left(c(t)\right)\,.
\end{eqnarray*}
\end{proof}
%%%%%%%%%%%%%%%%%%%%%%%%%%%%%%%%%%%%%%%%%%%%%%%%%%%%%%%%%%%%%%%%%%%%%%%%%%%%
%%%%%%%%%%%%%%%%%%%%%%%%%%%%%%%%%%%%%%%%%%%%%%%%%%%%%%%%%%%%%%%%%%%%%%%%%%%%%%%%%
\begin{remark}[{\sc Addition rules}]
\rm
The second order linear system~\eqref{eq:linearode}
of ordinary differential equations
is equivalent to the following 
first order linear system of ordinary
differential equations
\begin{eqnarray*}
\frac{d}{dt}
\left(
 \begin{array}{c}
  c(t)\\
\dot{c}(t)
 \end{array}
\right)
=
\left(
\begin{array}{cc}
 0 &\Identity\\
B  & 0
\end{array}
\right)
\left(
 \begin{array}{c}
  c(t)\\
\dot{c}(t)
 \end{array}
\right)
\end{eqnarray*}
with constant coefficients.
The solution of this system 
with initial values
$c(0)=v, \dot{c}(0)=w$ is given by
\begin{equation*}
\left(
 \begin{array}{c}
  c(t)\\
\dot{c}(t)
 \end{array}
\right)
=
 \exp\left(\left(
\begin{array}{cc}
 0 & \Identity\\
B & 0
\end{array}
\right) t\right)
\left(
 \begin{array}{c}
  v\\
w
 \end{array}
\right)
=
\left(
\begin{array}{cc}
 \coA(t) & \siA(t)\\
B \cdot \siA(t) & \coA(t)
\end{array}
\right)
\left(
 \begin{array}{c}
  v\\
w
 \end{array}
\right)\,.
\end{equation*}
Hence the curve $t \mapsto (c(t),\dot{c}(t))$
is the orbit of the point $(v,w) \in \R^n \oplus \R^n$
under a one-parameter group of linear transformations.

Since
\begin{eqnarray*}
 \exp\left(\left(
\begin{array}{cc}
 0 & \Identity\\
B & 0
\end{array}
\right)(t+s)\right)
=
\left(
\begin{array}{cc}
 \coA(t+s) & \siA(t+s)\\
B \cdot \siA(t+s) & \coA(t+s)
\end{array}
\right)
\\
=
\exp\left(\left(
\begin{array}{cc}
 0 & \Identity\\
B & 0
\end{array}
\right)t\right)
\cdot
\exp\left(\left(
\begin{array}{cc}
 0 & \Identity\\
B & 0
\end{array}
\right)s\right)
\\
=
\left(
\begin{array}{cc}
 \coA(t) & \siA(t)\\
B \cdot \siA(t) & \coA(t)
\end{array}
\right)
\cdot
\left(
\begin{array}{cc}
 \coA(s) & \siA(s)\\
B \cdot \siA(s) & \coA(s)
\end{array}
\right)
\end{eqnarray*}
we obtain the following
addition rules:
\begin{eqnarray}
\label{eq:addition-coA}
\coA(t+s)&=&\coA(t)\cdot\coA(s)+B \cdot \siA(t)\cdot\siA(s)\\
\siA(t+s)&=&\coA(t)\cdot\siA(s)+\siA(t)\cdot\coA(s)\,.
\end{eqnarray}
\end{remark}
%%%%%%%%%%%%%%%%%%%%%%%%%%%%%%%%%%%%%%%%%%%%%%%%%%%%%%%%%%%%%%%%%%%%%%%%%%%%%%%%
\begin{remark}[{\sc Roots of $B$}]
\rm
If there is a matrix $C\in \mnr,$ such that $B=C^2$
then the curve
$ c(t)= c[v,Cv](t)=\exp(tC)(v)$
satisfies
$
c(t+s)=\exp(sC) \,c(t),
$
i.e. the curve is the orbit of a point under a
one-parameter group of affine transformations
in $\R^n.$
\end{remark}
%%%%%%%%%%%%%%%%%%%%%%%%%%%%%%%%%%%%%%%%%%%%%%%%%%%%%%%%%%%%%%%%%
Let $c$ be a 
solution of the inhomogeneous
linear differential
equation
$\ddot{c}(t)=B c(t)+d,$ let $U \in \mnr $ be an invertible matrix and
$e$ a vector,
then $f(t)=U c(t)+e$ solves the differential equation
$\ddot{f}(t)=U B U^{-1} f(t)+ \left(\Identity-UBU^{-1}\right) e+U d.$ 
Therefore one
can assume that $B \in \mnr$ has already 
complex Jordan normal form
cf. \cite[Thm. 5.4.10]{Klingenberg1990}
or real Jordan normal form,
cf. \cite[Thm. 5.6.3]{Klingenberg1990}.
If an eigenvalue $\lambda$  of $A$ is not real then
the conjugate value $\overline{\lambda}$ is an eigenvalue,
too. In Section~\ref{sec:planar}
we discuss the possible real Jordan normal
forms for $n=2.$
For a complex number $\lambda$ the
Jordan block $J_m(\lambda)$ is given by:
\begin{equation}
\label{eq:jordanblock}
J_m(\lambda):=
\left.
\left(\begin{array}{cccccc}
\lambda & 1 & 0& 0&\ldots &0\\
0&\lambda &1 &0&\ldots &0\\
\ldots \\
\ldots \\
0&0& 0 &0&\ldots& 1\\
0&0&0&0&\ldots&\lambda
\end{array}
\right)\right\} m\,.
\end{equation}
The {\em complex Jordan normal form} of a real
matrix $B$ consists of Jordan blocks of the
form $J_m(\lambda), m\ge 1$ for an eigenvalue 
$\lambda \in \complexnumbers. $

The {\em real Jordan normal form} of a real matrix
$B$ consists of two different types of
Jordan blocks: For a real eigenvalue
$\lambda$ there are Jordan blocks of the
form  $J_m(\lambda), m\ge 1.$
If $\mu=\alpha+i\beta, \alpha\in \R,
\beta \in \R \setminus \{0\}$ is 
a non-real eigenvalue, then
for some $m\ge 1$ real Jordan blocks of the
form
\begin{equation}
J_{2m}\left(\alpha, \beta\right):=
\left.
\left(\begin{array}{cc}
J_m\left(\alpha\right) & -\beta \Identity\\
\beta \Identity & J_m\left(\alpha\right)
\end{array}
\right)\right\} 2m\,,
\end{equation}
occur.
The real Jordan block
$J_{2m}\left(\alpha,\beta\right)$ 
is always invertible. Hence 
the Jordan normal form of
a singular real matrix
contains a
nilpotent Jordan block of the form $J_m\left(0\right),
m\ge 1.$ 
Therefore
it 
is sufficient to discuss the following cases:
%%%%%%%%%%%%%%%%%%%%%%%%%%%%%%%%%%%%%%%%%%%%%%%%%%%%%%%%%%%%%%%%%%%%%%%%%%%%%%%%%
\begin{proposition}[\sc Inhomogeneous 
differential equation]

\label{pro:odeb}

For a matrix $B \in \mnr$ and a vector
$d\in \R^n$
we consider
the inhomogeneous linear
system of ordinary differential equations
(with constant coefficients) of second order:
\begin{equation}
\label{eq:linearodeb}
\ddot{c}(t)=B c(t)+d\,,
\end{equation}
and we consider the
one-parameter family of curves 
$\cas$ 
defined by Equation~\eqref{eq:oneparameter}.
Then we have
\be
\label{eq:csA}
\cas(t)= \as\;
c(t) + \bs
\,,
\ee
with
$$\as=
\frac{1}{2}\left\{\Identity+\coA(s) \right\}$$
i.e.~Equation~\eqref{eq:hauptgleichung} and 
Equation~\eqref{eq:A(s)} hold .\\
We consider three cases:

\begin{itemize}

%%%%%%%%%%%%%%%%%%%%%%%%%%%%%%%%%%%%%%%%%%%%%%%%%%%%%%%%%%%%%%%%%%%%%%%%%%%%%%%%%
\item[(a)] If there is a vector $d_*$ such that $d=B \cdot d_*\,,$ then
for $v, w \in \R^n$
the unique solution of Equation~\eqref{eq:linearodeb}
with intitial values $c(0)=v,\dot{c}(0)=w$ 
is given 
by
\begin{equation}
\label{eq:ctvwb}
c(t)=c[v,w](t)=\coA(t)\left(v+d_*\right)+\siA(t)(w)-d_*\,.
\end{equation}
and
\begin{eqnarray*}
b(s)=\halb
\left\{\coA(s)-\Identity \right\}(d_*)
\end{eqnarray*}
%%%%%%%%%%%%%%%%%%%%%%%%%%%%%%%%%%%%%%%%%%%%%%%%%%%%%%%%%%%%%%%%%%%%%%%%%%%%%%%%%
\item[(b)]
Let $B$ be the  
nilpotent {\em Jordan block matrix:}
$B=N_n:=J_n(0);$ cf. Equation~\ref{eq:jordanblock}.
Hence $N_n(e_1)=0$
and $N_n(e_j)=e_{j-1}$ for $j\ge 2$
for a basis $e_1,e_2,\ldots,e_n$ of $\R^n.$
Then the unique solution $c_*=c_*(t)$ of 
\begin{equation}
\label{eq:cstar}
\ddot{c_*}(t)=N_n c_*(t)+e_n, c_*(0)=\dot{c_*}(0)=0
\end{equation}
is given by
$$c_*(t)=\left(\frac{t^{2n}}{(2n)!},\ldots,\frac{t^2}{2!}\right).$$
$d \in \R^n$ can be written as
$d=\left(d_1,d_2,\ldots,d_n\right)
=B \cdot d_*+d_ne_n$
with $d_*=\left(0,d_1,d_2,\ldots,d_{n-1}\right).$
And
the unique solution of Equation~\eqref{eq:linearodeb}
with intitial values $c(0)=v,\dot{c}(0)=w$ is given 
by
\begin{eqnarray}
c(t)=
\coA(t)\left(v+d_*\right)+\siA(t)(w)-d_*+d_nc_*(t)
\label{eq:ctsol}
\\
\bas=
\halb
\left\{\coA(s)-\Identity\right\}(d_*)+
\halb\,d_n c_*(s)\,.
\label{eq:ctsol2}
\end{eqnarray}

\item[(c)]
If $B=0$ then the unique solution of 
Equation~\eqref{eq:linearodeb} with initial values
$v=c(0),w=\dot{c}(0)$ and the one-parameter family
$c_s(t)$ is given by:
\ben
c(t)=\cvw(t)=\frac{t^2}{2}\,d+wt+v
\,;\,
c_s(t)=c(t)+\viertel\,s^2 d
\een
and
$$
b\textbf{}(s)=\viertel\,s^2 d\,.
$$
\end{itemize}
%%%%%%%%%%%%%%%%%%%%%%%%%%%%%%%%%%%%%%%%%%%%%%%%%%%%%%%%%%%%%%%%%%%%%%%%%%%%%%%%%
\end{proposition}
%%%%%%%%%%%%%%%%%%%%%%%%%%%%%%%%%%%%%%%%%%%%%%%%%%%%%%%%%%%%%%%%%%%%%%%%%%%%
\begin{proof}
The curve
\ben
%\label{eq:ctvwb}
c(t)=\cvw(t)=\coA(t)\left(v+d_*\right)+\siA(t)(w)-d_*
\een
with $c(0)=v,\dot{c}(0)=w$ 
satisfies
by Proposition~\ref{pro:derivative-coA}:
\ben
%\label{eq:cvwbstar}
\ddot{c}(t)=B\left\{\coA(t)\left(v+d_*\right)+\siA(t)(w)\right\}=
B\left(c(t)+d_*\right)=Bc(t)+B(d_*)\,.
\een
This already proves Equation~\eqref{eq:linearodeb} in case (a),
or rather
Equation~\eqref{eq:ctvwb}.
The addition rules Equation~\eqref{eq:addition-coA} show:
\begin{eqnarray}
\cas(t)=\viertel\left\{ c(t-s)+2c(t)+ c(t+s)\right\}
\nonumber
\\
=
\viertel \, \left[\coA(t-s)+\coA(t+s)\right]\left(v+d_*\right)
+\halb \, \coA(t)\left(v+d_*\right)+
\nonumber
\\
\viertel \, \left[\siA(t-s)+\siA(t+s)\right](w)
+\halb \,\siA(t)(w)
-d_*
\nonumber
\\
\label{eq:csa}
=\halb
\left\{\Identity+ \coA(s)\right\}
\left[\coA(t)(v+d_*)+\siA(t)(w)\right]-d_*\\
\nonumber
= \Aas (c(t))+b(s)
\end{eqnarray}
with
\begin{equation*}
\Aas=\halb
\left\{
\coA(s)+\Identity 
\right\}
 \,;\,
\bas=\halb \left\{\coA(s)-\Identity \right\}(d_*)\,.
\end{equation*}

\smallskip

(b)
Now let $B=N_n.$ One checks that $c_*(t)$ given by
Equation~\eqref{eq:cstar} is the unique solution of
Equation~\eqref{eq:linearodeb} with $c(0)=\dot{c}(0)=0$
and $b=e_n.$
Then for any solution $c=c(t)$ of Equation~\eqref{eq:linearodeb}
the curve $c(t)-d_n c_*(t)$
is a solution of Equation
\ben
\ddot{e}(t)=B \left(e(t)+d_*\right)=
B e(t)+\left(d_1,\ldots,d_{n-1},0\right)\,.
\een
Hence we conclude from Part (a) that 
$$
e(t)=\coA(t)\left(v+d_*\right)+\siA(t)(w)-d_*.
$$
This implies that $c(t)$ satisfies Equation~\eqref{eq:ctsol}.
Since
\ben
\viertel\left\{
c_*(t-s)+2c_*(t)+ c_*(t+s)
\right\}=
\halb
\left\{ \coA(s)+\Identity\right\}\left(c_*(t)\right)
\een
together with Equation~\eqref{eq:csa} proves 
Equation~\eqref{eq:hauptgleichung} and
Equation~\eqref{eq:ctsol2}.

\smallskip

(c)
For $B=0$ Equation~\eqref{eq:linearodeb} is $\ddot{c}(t)=d,$ hence we obtain
$c(t)=d t^2/2+wt+v$ 
. Furthermore $\coA(s)=\Identity=\as$ and
$\cas(t)=\halb\left\{ 
c(t-s)+2c(t)+ c(t+s)
\right\}
=c(t)+\viertel s^2 d \,.$
\end{proof}
%%%%%%%%%%%%%%%%%%%%%%%%%%%%%%%%%%%%%%%%%%%%%%%%%%%%%%%%%%%%%%%%%%%%%%%%%%%
%%%%%%%%%%%%%%%%%%%%%%%%%%%%%%%%%%%%%%%%%%%%%%%%%%%%%%%%%%%%%%%%%%%%%%%
\begin{remark}[\sc Wave equation]

\label{rem:wave}
\rm
Let $c=c(t)$ be a smooth curve 
with the one-parameter family of curves
$c_s$ defined by 
Equation~\eqref{eq:oneparameter}.
Then 
\ben
%\label{eq:casst}
\frac{\partial^2 \cas (t)}{\partial s^2}=
\viertel \, \ddot{c}(t-s)+\viertel \,\ddot{c}(t+s)=
\frac{\partial^2 \cas (t)}{\partial t^2}
-\halb \, \ddot{c}(t)\,.
\een
Then 
$\tcas(t):=\cas(t)-\halb \,c(t)=
\left\{c(t-s)+c(t+s)\right\}/4$
defines 
a solution of the 
{\em wave equation:} 
\ben
\frac{\partial^2 \tcas (t)}{\partial s^2}=
\frac{\partial^2 \tcas (t)}{\partial t^2}
\een
%$
%\tcas(t)=
%\left(A_{\alpha}(s)+(2\alpha-1)\Identity\right)c(t)+\bas\,.
%$
with initial conditions 
$\widetilde{c}_0(t)=c(t)/2$ and 
\begin{equation*}
\left.\frac{\partial \tcas (t)}{\partial s}\right|_{s=0}=0\,.
\end{equation*}
If the smooth curve $c$ is a solution of 
Equation~\eqref{eq:linearodeb}
then 
Proposition~\ref{pro:odeb}
implies:
\begin{eqnarray*}
\tcas(t)=
\left(\Aas-\Identity/2\right) 
\left(\widetilde{c}_{0}(t)\right)+\bas
=\coA(s)\left(
\widetilde{c}_0(t)\right)+\bas\,.
\end{eqnarray*}
Hence 
during the evolution $s\mapsto \tcas$ 
the 
affine form of the curve $\widetilde{c}=\widetilde{c_0}$ 
is preserved. This motivates the notion
{\em soliton} for these
solutions of the wave equation.
%i.e. $\tcas$  defines an affinely invariant solution of
%the wave equation.
\end{remark}
In the sequel we study which invertible matrices $D$ can
be written in the form $\left(\Identity+\coA(s)\right)/2$
for some real matrix $B.$
Since for any invertible matrix $U$ we have
$\co_{UBU^{-1}}(s)=U \coA(s) U^{-1}$
it is sufficient to check the possible Jordan normal forms
of $\coA(s)$ for the different Jordan blocks $B.$
It turns out that a large class of invertible matrices $D$
can be 
written in this form. 
%%%%%%%%%%%%%%%%%%%%%%%%%%%%%%%%%%%%%%%%%%%%%%%%%%%%%
\begin{proposition}[\sc The 
image of $(B,s)\mapsto \left(\Identity+\coA (s)\right)/2$]
\label{pro:coA}
Consider the mapping
$f: (B,s) \in M(n;\R)\times \R \mapsto 
f(B,s):=\left(\Identity+\coA(s)\right)/2.$
An invertible matrix $D \in \mnr$ 
lies in the image of $f$
if and only if
the (real or complex) Jordan normal form 
of $D$ 
satisfies the following:
If $\lambda <0$ is a real and negative 
eigenvalue of $D$ and if for some
$m\ge 1$ the Jordan block $J_m(\lambda)$
occurs in the Jordan normal form, then
the number of Jordan blocks $J_m(\lambda)$
in the Jordan decomposition of $D$ is even.

\end{proposition}
%%%%%%%%%%%%%%%%%%%%%%%%%%%%%%%%%%%%%%%%%%%%%%%%%
In particular a diagonizable matrix $D$ 
can be written as:
$D=f(B,s)=\left(\Identity+\coA(s)\right)/2$ 
for some $(B,s) \in \mnr\times \R$
if and only if the eigenspaces
of all real and negative eigenvalues
are even-dimensional.

%%%%%%%%%%%%%%%%%%%%%%%%%%%%%%%%%%%%%%%%%%%%%%%%%
%%%%%%%%%%%%%%%%%%%%%%%%%%%%%%%%%%%%%%%%%%%%%%%%%
\begin{proof}
We compute for the possible complex Jordan normal
forms $J$ of a real matrix $B$ the complex 
Jordan normal
form of $\co_J(s)$ resp. $f(J,s).$
Since the matrix $D=\left(\Identity+\coA(s)\right)/2$
is supposed to be invertible we
exclude $-1$ as eigenvalue of $\coA(s).$

\smallskip

(a) If $J=\lambda \cdot\Identity, \lambda \in \R,$
then
$f(J,s)=(1+\cos_{\lambda}(s))/2 \cdot \Identity.$
Note that
$\left\{(1+\cos_{\lambda}(s))/2\,;\, \lambda, s \in \R \right\}=
\{ x \in \R\,;\, x\ge 0\}.$
Hence $f(J,s)$
is a diagonal matrix
with non-negative real eigenvalues.

\smallskip

(b) Assume the matrix $B$ contains
a complex Jordan block $J=u \cdot \Identity_m, u\in \complexnumbers \setminus \R,$  then also $\overline{J}=\overline{u}\cdot \Identity$
is a (distinct) Jordan block, i.e.~there is an
even-dimensional
invariant subspace on which the Jordan normal
form is given by
$J \oplus \overline{J}=
u \cdot \Identity_m \oplus \overline{u}\cdot \Identity_m.$
%If $u \in \R$ we assume in addition that
%$m$ is even. 
Then
$f\left(J\oplus \overline{J},s\right)=\left(1+\cosh(ws)\right)/2
\cdot \Identity \oplus
\left(1+\cosh(\overline{w}s)\right)/2
\cdot \Identity,$
here $w \in \complexnumbers, u=w^2.$
Note that  $\cosh: \complexnumbers \rightarrow \complexnumbers$ is surjective. If $\cosh(ws)\in \R$ the corresponding
eigenspace  is even-dimensional.

We conclude from (a) and (b):
A real and invertible  matrix $D$ which is diagonalizable over
$\complexnumbers$ can be written in the form
$D=f(B,s)$
for a 
matrix $B$ diagonalizable over $\complexnumbers$
and a real number $s$
if and only if
it does not have a real and negative eigenvalue
$\lambda$ with an odd-dimensional eigenspace.

\smallskip

(c) Now assume that 
$J=J_m(u)$ with $m\ge 2.$
For $u=0:$ we obtain for $N_m=J_m(0):$
\begin{equation}
\co_{N_m}(s)= \Identity_m+\frac{s^2}{2!}N_m+\ldots+
\frac{s^{2m-2}}{(2m-2)!}N_m^{m-1}\,.
\end{equation}
Note that the matrix $N_m^k$ satisfies
$N_m^k(e_{l+k})=e_{l}, l=1,2,\ldots, m-k, N_m^k(e_l)=0,
l=1,2,\ldots,k-1.$
Since for $s\not=0$
$${\rm rank} \left(\co_{N_m}(s)-\Identity_m
\right)=m-1$$
we obtain: For $s\not=0$ the matrix
$\co_{N_m}(s)$ is conjugate to the Jordan block
$J_m(1)$
and $f(N_m,s)$ is also conjugate to the Jordan
block $J_m(1).$

For $u\not=0$ we obtain:
\begin{eqnarray}
\label{eq:cou}
\co_J(s)=\sum_{k=0}^{\infty} \frac{s^{2k}}{(2k)!}J^k=
\sum_{k=0}^{\infty}\frac{s^{2k}}{(2k)!}
\sum_{j=0}^{\min(m,k)}
\left(
\begin{array}{c}
k\\
j
\end{array}
\right)
u^{k-j}N_m^j
\\ \nonumber
= \sum_{k=0}^{\infty}\frac{s^{2k}}{(2k)!}u^k \cdot \Identity_m + 
\frac{s}{2}
\sum_{k=1}^{\infty}\frac{s^{2k-1}}{(2k-1)!}
u^{k-1}\cdot N_m+
\\ \nonumber
 + 
\alpha_2(s) N^2_m+\ldots +
\alpha_{m-1}(s)\cdot N_m^{m-1}
\nonumber
\end{eqnarray}
for some $\alpha_2(s),
\ldots,\alpha_{m-1}(s) \in \complexnumbers.$
If $u=\lambda\in \R$ we have
\begin{equation*}
\co_J(s)=\cos_{\lambda}(s)\cdot \Identity_m + 
\frac{s}{2}\,\sin_{\lambda}(s)\cdot N_m +
\alpha_2 N_m^2+\ldots+
\alpha_{m-1} N_m^{m-1}\,.
\end{equation*}
Note that $\{\cos_{\lambda}(s)\,;\,\lambda, s \in \R\}
=\{x\in \R, x\ge -1\}$ and that
$\sin_{\lambda}(s)\not=0$ whenever $\cos_{\lambda}(s)
\not=\pm 1.$
Hence for $\cos_{\lambda}(s)\not=\pm 1$ the
matrix $\co_J(s)$ satisfies:
$${\rm rank} \left(\co_J(s)-\cos_{\lambda}(s)
\cdot \Identity_m\right)=m-1$$
i.e.~the matrix
$f\left(J_m(\lambda),s\right)$ is conjugate to 
the 
Jordan block 
$J_m\left(\left(1+\cos_{\lambda}(s)\right)/2\right).$
Thus we have shown that a Jordan block
$J_m(\mu)$ with real $\mu\not=0$
and $m\ge 2$
can be written in the form 
$f\left(J_m\left(\lambda\right),s\right)$
for some real $\lambda$ and some $s \in \R$
up to conjugacy
if and only if $\mu$ is positive.

\smallskip

(d) Now assume that
the real matrix $B$ contains  a complex
Jordan block $J=J_m(u),
u \in \complexnumbers \setminus \R.$ 
Then let $w \in \complexnumbers \setminus \R$
be a square root, i.e.~$u=w^2.$
Hence also 
$\overline{J}=J_m(\overline{u})$
is a 
(distinct) Jordan block of $B,$
or rather
there is an invariant subspace
on which the Jordan normal form of $A$
is given by
$J \oplus \overline{J}=
J_m(u)\oplus J_m(\overline{u}).$
We conclude from Equation~\ref{eq:cou}:
\ben
\co_{J}(s)
= \cosh(ws) \cdot \Identity_m + 
\frac{s}{2 w} \sinh(ws)
N_m 
+ \alpha_2(s) N^2_m
+\ldots 
+\alpha_{m-1}(s)\cdot N_m^{m-1}
\een
for some $\alpha_2(s),
\ldots,\alpha_{m-1}(s) \in \complexnumbers.$
If $\cosh(ws)\not=\pm 1$ we have
$\sinh(ws)\not=0,$ 
i.e. 
$${\rm rank} \left(\co_{J}(s)-\cosh(ws)\cdot \Identity_m
\right)=m-1$$
or rather
the matrix $\co_{J}(s)$ is conjugate
to a Jordan block
$J_m\left(\cosh(ws)\right)$
and
$f\left(J \oplus\overline{J}\right)$
is conjugate to 
$J_m\left(\left(1+\cosh(ws)\right)/2\right)
\oplus
J_m\left(\left(1+
\cosh(\overline{w}s)\right)/2\right)\,.$
We conclude from (c) and (d) that an
invertible and
real matrix $D$ whose Jordan normal form
contains
a Jordan block 
$J=J_m(v), m\ge 2$
can be written in the form
$f(B,s)$
for a real matrix $B$ and $s \in \R$
if and only if
the following holds:
If $v$ is real and negative $v=\lambda<0$ then the 
number of
Jordan blocks $J_m\left(\lambda\right)$ in the Jordan decomposition
of $B$ is even.
\end{proof}
%%%%%%%%%%%%%%%%%%%%%%%%%%%%%%%%%%%%%%%%%%%%%%%%%%%%%%%%%%%%%%
%%%%%%%%%%%%%%%%%%%%%%%%%%%%%%%%%%%%%%%%%%%%%%%%%%%%%%%%%%%%%%
%%%%%%%%%%%%%%%%%%%%%%%%%%%%%%%%%%%%%%%%%%%%%%%%%%%%%%%%%%%%%%
%%%%%%%%%%%%%%%%%%%%%%%%%%%%%%%%%%%%%%%%%%%%%%%%%%%%%%%%%%%%%%

%%%%%%%%%%%%%%%%%%%%%%%%%%%%%%%%%%%%%%%%%%%%%%%%%%%%%%%%%%%%%%%%%%%%
%%%%%%%%%%%%%%%%%%%%%%%%%%%%%%%%%%%%%%%%%%%%%%%%%%%%%%%%%%%%%%%%%%%%
%%%%%%%%%%%%%%%%%%%%%%%%%%%%%%%%%%%%%%%%%%%%%%%%%%%%%%%%%%%%%%%%%%%%
\section{Discrete Curve Shortening}
%%%%%%%%%%%%%%%%%%%%%%%%%%%%%%%%%%%%%%%%%%%%%%%%%%%%%%%%%%%%%%%
%%%%%%%%%%%%%%%%%%%%%%%%%%%%%%%%%%%%%%%%%%%%%%%%%%%%%%%%%%%%%%
%%%%%%%%%%%%%%%%%%%%%%%%%%%%%%%%%%%%%%%%%%%%%%%%%%%%%%%%%%%%%%
%%%%%%%%%%%%%%%%%%%%%%%%%%%%%%%%%%%%%%%%%%%%%%%%%%%%%%%%%%%%%%
An (infinite) {\em polygon} 
$x=\left(x_j\right)_{j\in \z}$ in  $\R^n$ is defined by
its {\em vertices} $x_j\in \R^n.$ 
We call $\po=\po(\R^n)$ the vector space of these polygons.
We can identify the polygon $x$ with the piecewise linear curve
$x:\R \rightarrow \R^n$ which is a straight line on any interval
$[j,j+1]$ and satisfies $x(j+u)=(1-u)x_j+u x_{j+1}$
for any $u \in [0,1], j\in \z.$
If there is a positive number $N$ such that
$x_{j+N}=x_j$ for all $j\in \z,$ then we call
the polygon $x$  {\em closed} 
or rather
{\em periodic}
with $N$ vertices 
or rather
of period $N.$ In this case we can
identify the index set with $\z_N=\z/(N \cdot\z).$ We denote the 
set of closed polygons with $N$ vertices by $\poN=\poN(\R^n).$
%%%%%%%%%%%%%%%%%%%%%%%%%%%%%%%%%%%%%%%%%%%%%%%%%%%%%%%%%%%%%%%
The {\em midpoint mapping} is given by
\ben
M: x \in \po(\R^n) \longmapsto M(x)\in \po(\R^n)\,;\,
\left(M(x)\right)_j:=\frac{1}{2}\left(x_j+x_{j+1}\right).
\een
%%%%%%%%%%%%%%%%%%%%%%%%%%%%%%%%%%%%%%%%%%%%%%%%%%%%%%%%%%%%%%%
For a closed polygon $x \in \poN (\R^n)$ its {\em length}
is given by\\
$
L\left(x\right):=\sum_{j=0}^{N-1}\left\|x_{j+1}-x_j\right\|\,,
$
here $\left\|.\right\|$ denotes the Euclidean norm.
\smallskip

The triangle inequality implies
the following curve shortening property of the midpoint mapping
in the general case:
\begin{eqnarray*}
L\left(M(x)\left|\left[j,j+m\right]\right.\right)&=&
\sum_{k=1}^{m}\left\|\left(M(x)\right)_{j+k}-
\left(M(x)\right)_{j+k-1}\right\|\\
&\le&
L\left(x\left| \left[j+1/2,j+m+1/2\right]\right.\right)
\end{eqnarray*}
with equality if and only if the points $x_j,x_{j+1},\ldots,x_{j+m}, x_{j+m+1}$ lie
on a straight line.
Thus
the midpoint mapping $M$ is  {\em length decreasing} on closed polygons, i.e.~for all closed polygons $x$ we have:
$L\left(M(x)\right)\le L(x)$
with equality if and only if $x$ is constant, i.e.~$L(x)=0.$
%%%%%%%%%%%%%%%%%%%%%%%%%%%%%%%%%%%%%%%%%%%%%%%%%%%%%%%%%%%%%%%
%%%%%%%%%%%%%%%%%%%%%%%%%%%%%%%%%%%%%%%%%%%%%%%%%%%%%%%%%%%%%%
%%%%%%
We modify the midpoint mapping as follows:
%%%%%%%%%%%%%%%%%%%%%%%%%%%%%%%%%%%%%%%%%%%%%%%%%%%%%%%%%%%%%%%%%%
\begin{definition}[\sc Curve shortening process]

\label{def:selfsimilar}

We introduce the following mapping $T: \po (\R^n)
\rightarrow  \po (\R^n):$
\ben
%\label{eq:T}
\left(T(x)\right)_j=
%\left(M^2(x)\right)_{j-1}=
\viertel\left\{
x_{j-1}+2x_j+x_{j+1}
 \right\}\,.
\een

(a)
We call a polygon $x=(x_j)_{j\in \z}$ {\em affinely invariant}
under the mapping $T$ 
or rather an {\em (affine) soliton} of the curve shortening process
$T$)
if there is an affine map
$(A,b), A\in Gl\left(n,\R\right), b\in \R^n$ such that for all $j\in \z:$
\begin{equation}
\label{eq:txj}
T(x)_j=
\viertel\left\{
x_{j-1}+ 2x_j+x_{j+1}
\right\}=
A\left(x_j\right)+b.
\end{equation}
(b)
We call a smooth curve $c:\R \rightarrow \R^n$
{\em affinely invariant} under the mapping $T$
(or rather
 an {\em (affine) soliton} of the curve shortening process
$T$) if there is some $\epsilon>0$ such that 
there is a one-parameter family 
 $s\in (-\epsilon,\epsilon)
 \mapsto \left(\as,\bs\right)
\in Gl(n,\R)\times \R^n$
of affine 
maps such that
\be
\label{eq:axj}
c_s(t)=\viertel\left\{
c(t-s)+ 2c(t)+ c(t+s)
\right\}
=
\as c(t)+\bs
\ee
for all $s\in (0,\epsilon),t\in \R.$
\end{definition}
%%%%%%%%%%%%%%%%%%%%%%%%%%%%%%%%%%%%%%%%%%%%%%%%%%%%%%%%%%%
%%%%%%%%%%%%%%%%%%%%%%%%%%%%%%%%%%%%%%%%%%%%%%%%%%%%%%%%%%%
It is obvious that these notions are affinely invariant. 
For $a \in \R, s \in (0,\epsilon)$ 
the polygon $x=x(a,s)$ with
$x(a,s)_j=c(a+sj), j \in \z$
lying on a smooth curve $c:\R \longrightarrow \R^n$
is a soliton of the curve shortening process $T$
if the curve $c$ is also a soliton of the 
corresponding process $T$ on curves.
On the other hand: A polygon 
$x=\left(x_j\right)_{j\in\z}$
which is a soliton of the curve shortening
process $T$ satisfying Equation~\eqref{eq:txj}
can be obtained from a smooth curve, which is
a soliton as defined in Equation~\eqref{eq:axj}
if and only if $A$ can be written in the
form $A=\left(\Identity+\coA(s)\right)/2.$
In Proposition~\ref{pro:coA} 
the Jordan normal form of these matrices are
classified.

\begin{remark}[\sc Eigenpolygons of $T$]
\rm
If we consider closed polygons
then
the midpoint mapping defines a linear map
$M: \po_N \longrightarrow \po_N$
on the $(n\cdot N)$-dimensional vector space $\po_N,$ and one can use a
decomposition into eigenspaces, 
cf. \cite{Schoenberg1950} and \cite{Berlekamp1965}.
The matrix is in particular {\em circulant.} 
For $n=2$ one can identify the complex numbers
$\complexnumbers$ with $\R^2,$ then the 
eigenvalues are
$\lambda_k=\left(1+\exp(2\pi i k/N)\right)/2\,, k=0,1,\ldots,N-1.$
The corresponding eigenvectors are the polygons
\be
\label{eq:berlekamp}
z^{(k)}=\left(z^{(k)}_1,\ldots,z^{(k)}_N\right)\in \complexnumbers^N\enspace;\enspace
z^{(k)}_j=\exp\left(2\pi i j k/N\right)\,
\ee
for $k=0,1,\ldots,N-1.$
Then the linear map $T:\po_N \longrightarrow \po_N$
is also circulant and has the form
\ben
T=\frac{1}{4}
\left(
\begin{array}{ccccccccc}
2 &1&0&0&\ldots&0&0&0&1\\
1&2&1&0 & \ldots&0&0&0&0\\
0&1&2&1&\ldots&0&0&0&0\\
\ldots\\
\ldots\\
\ldots\\
0&0&0&0&\ldots&1&2&1&0\\
0&0&0&0&\ldots&0&1&2&1\\
1&0&0&0&\ldots&0&0&1&2
\end{array}
\right)
\een
with eigenvalues
\ben
\mu_k=\mu_{N-k}=
\frac{1}{2}\left(1+\cos\left(\frac{2\pi k}{N}\right)\right),
 k=0,1,\ldots,N-1
\een
and eigenvectors $z^{(k)}$ given by Equation~\eqref{eq:berlekamp}.
Note that all polygons $z^{(k)}$ given by
Equation~\eqref{eq:berlekamp} lie on the
unit circle $c(t)=\exp(2\pi i t)$ and
$z^{(k)}_j=c\left(kj/N\right).$ The eigenpolygons $z^{(k)}$ are 
solitons on which the map $T$ is a homothety.
Although the curve $c$ is simple and convex
the polygons $z^{(k)}$ form regular $N$-gons which 
are simple and convex only for $k=1,N-1,$ 
cf.~\cite[Fig.5]{Berlekamp1965}.
\end{remark}
%%%%%%%%%%%%%%%%%%%%%%%%%%%%%%%%%%%%%%%%%%%%%%%%%%%%%%%%%%%%%%%%%%%

%%%%%%%%%%%%%%%%%%%%%%%%%%%%%%%%%%%%%%%%%%%%%%%%%%%%%%%%%%%%%%%%%
%%%%%%%%%%%%%%%%%%%%%%%%%%%%%%%%%%%%%%%%%%%%%%%%%%%%%%%%%%%%%%%%%%%
\begin{proposition}[\sc assignment 
matrix $A$ $\rightarrow $ polygon]

\label{pro:uniqueness}

Let $(A,b): x \in \R^n \longmapsto Ax+b \in\R^n$ be an affine map
and $u,v \in \R^n$ be two points in $\R^n,$
and $j_0\in \z.$
Then there is a unique polygon $x \in \po(\R^n)$ with $x_{j_0}=u,x_{j_{0}+1}=v$
which 
is affinely invariant
(with respect to $A$ and $b$)
under the mapping $T.$
\end{proposition}
%%%%%%%%%%%%%%%%%%%%%%%%%%%%%%%%%%%%%%%%%%%%%%%%%%%%%%%%%%%%%%%%%%%%%
%%%%%%%%%%%%%%%%%%%%%%%%%%%%%%%%%%%%%%%%%%%%%%%%%%%%%%%%%%%%%%%%%%%%
%%%%%%%%%%%%%%%%%%%%%%%%%%%%%%%%%%%%%%%%%%%%%%%%%%%%%%%%%%%%%%%%%%
\begin{proof}
If $x\in \po(\R^n)$ is affinely
invariant under the mapping $T$
we have
\begin{equation}
\left(T(x)\right)_j=\viertel
\left\{
x_{j-1}+
2x_j+  x_{j+1}
\right\}
=
A\left(x_j\right)+b\,.
\end{equation}
Hence the sequence $\left(x_j\right)$ with $x_{j_0}=u,x_{j_{0}+1}=v$
is uniquely determined by
the recursion formulae
\begin{equation}
x_{j+1}=2\left(2 A-\Identity\right)(x_j)-x_{j-1}+4b
\,;\,
x_{j-1}=2\left(2A-\Identity\right)(x_j)-x_{j+1}+4b\,.
\end{equation}
\end{proof}
%%%%%%%%%%%%%%%%%%%%%%%%%%%%%%%%%%%%%%%%%%%%%%%%%%%%%%%%%%%%%%%%%%%%%%%%%%%%%%%%%%%%%%%%%%%%%%%%%%%%%
%%%%%%%%%%%%%%%%%%%%%%%%%%%%%%%%%%%%%%%%%%%%%%%%%%%%%%%%%%%%%%%%%%%%%%%%
For a given smooth curve $c:\R \rightarrow \R^n$ 
we define the
one-parameter family $c_s: \R \rightarrow \R^n$
by Equation~\eqref{eq:oneparameter}.
%we can define an
%one-parameter family $c_{s}:\R \rightarrow \R^n$ as follows:
%$$
%\label{eq:cs2}
%c_{s}(t):=\viertel\left\{c(t-s)+2 c(t)+ c(t+s)
%\right\}
%$$
%}
The curves $c_{s}$ are obtained from $c=c_{0}$ by applying the 
mapping $T$ as follows.
For arbitrary $a \in \R, s>0$ let $x=x(a,s)$ 
be the polygon
$x_j=x_j(a,s):=c\left(a+js\right), j\in \z.$
Then
\ben
c_{s}\left(a+js\right)=\left(T(x)\right)_j=
\viertel\left\{
c(a+(j-1)s)+
2c(a+js)+ c(a+(j+1)s)
\right\}\,.
\een
Hence the vertices $\left(T\left(x(a,s)\right)\right)_j$ of the 
image $T\left(x(a,s)\right)$ of the polygon
$x=x(a,s)$ 
under the mapping $T$
lie on the curve $c_{s},$ 
or rather
 the curve $c_{s}$ is formed 
by the images $T\left(x\left(a,s\right)\right)$ 
of polygons of the form $x=x(a,s)$ on the curve $c.$
%%%%%%%%%%%%%%%%%%%%%%%%%%%%%%%%%%%%%%%%%%%%%%%%%%%%%%%%%%%%%%%%%%%%%%%%%%%%
\begin{theorem}[\sc solitons as solutions of an ode]
\label{thm:characterization}
Let
$c: \R \rightarrow \R^n$ be a smooth curve such that the
one-parameter family
defined by Equation~\eqref{eq:oneparameter}
defined for all $s \in (-\epsilon,\epsilon)$ 
and some $\epsilon >0$
satisfies:
$$ c_{s}(t)=\as(c(t))+\bs $$
for all $t \in \R, -\epsilon<s<\epsilon$
and for a smooth
one-parameter family $s \mapsto \as\in Gl(n,\R)$ of linear isomorphisms 
and a smooth curve $s \mapsto b(s)\in \R^n.$
I.e.~the curve $c$ is {\em affinely invariant} under the mapping $T,$
cf. Definition~\ref{def:selfsimilar} (b).
Assume in addition that for some $t_0\in \R$ the 
vectors 
$\dot{c}(t_0),\ddot{c}(t_0),\ldots,c^{(n)}(t_0)$ 
are linearly independent.

Then the curve $c$ is the unique solution
of the differential equation
\begin{equation}
\label{eq:ddotc}
\ddot{c}(t)=B c(t)+d
\end{equation}
with initial conditions $v=c(0),w=\dot{c}(0).$
The functions $s\mapsto \as\in M\left(n,\R\right),s\mapsto \bs\in \R^n$
satisfy the differential equations
\begin{equation}
\label{eq:a2s}
A''(s)=\asa B\,;\, 
b''(s)=\asa d
\end{equation}
with $B=2 A''(0), d=2b''(0)$
and initial conditions $A(0)=\Identity, b(0)=0,
A'(0)=0, b'(0)=0.$ In particular, we have
$\as=\left(\Identity +\coA(s)\right)/2,$ 
cf. Proposition~\ref{pro:odeb}.
The 
possible Jordan normal forms of 
the matrices $A(s)$ are given in
Proposition~\ref{pro:coA}.
\end{theorem}
%%%%%%%%%%%%%%%%%%%%%%%%%%%%%%%%%%%%%%%%%%%%%%%%%%%%%%%%%%%%%%%%%%%%%%%%%%%%
For a one-parameter family 
$s \mapsto \cas$ of curves or
a one-parameter family $ s\mapsto \as, s\mapsto \bs$
of matrices $\as$ or vectors $\bs$ we denote
differentiation with respect to $s$ by $'.$
Differentiation with respect to the curve parameter $t$ of
a one-parameter family of curves $c_s(t)$ or a single curve
$t \mapsto c(t)$ is denoted by
$\dot{c}$.

Note that for the {\em potential}
\begin{equation}
 \label{eq:potential}
 U(x)=-\frac{1}{2} \langle B x,x\rangle-\langle d,x\rangle
\end{equation}
we can write instead of Equation~\eqref{eq:ddotc} the following
Equation:
\begin{equation}
 \ddot{c}(t)=-\grad \,U\left(c(t)\right)\,.
\end{equation}
It follows that the function $$
f(t):=
\left\|\dot{c}(t)\right\|^2-\frac{1}{2} \langle B c(t),c(t)\rangle-
\langle d,c(t)\rangle $$
is constant for any solution of Equation~\eqref{eq:ddotc}.

%%%%%%%%%%%%%%%%%%%%%%%%%%%%%%%%%%%%%%%%%%%%%%%%%%%%%%%%%%%%%%%%%%%%%
\begin{proof}
Let 
\begin{equation}
\label{eq:cs2}
\cas(t)=\Aas c(t)+\bas=
\viertel\left\{
c(t-s)+2 c(t)+c(t+s)
\right\}
\,.
\end{equation}
%%%%%%%%%%%%%%%%%%%%%%%%%%%%%%%%%%%%%%%%%%%%%%%%%%%%%%%%%
For $s=0$ we obtain
$
c(t)=c_{0}(t)=A(0)c(t)+b(0)
$
for all $t \in \R,$ or rather
$
%\label{eq:a01}
\left(A(0)-\Identity\right)(c(t))=-b(0)
$
for all $t.$ We conclude that
\be
\label{eq:a02}
\left(A(0)-\Identity\right)\left(c^{(k)}(t)\right)=0
\ee
for all $k\ge 1.$
Since for some $t_0$ the vectors
$\dot{c}(t_0),\ddot{c}(t_0),\ldots,c^{(n)}(t_0)$
are linearly independent by assumption we conclude
from 
Equation~\eqref{eq:a02}:
$A(0)=\Identity,b(0)=0.$
Equation~\eqref{eq:cs2} implies for $k\ge 1:$
\ben
\as c^{(k)} (t)=
\viertel\left\{
c^{(k)}(t-s)+2c^{(k)}(t)+c^{(k)}(t+s)
\right\}\,.
\een
and hence
\ben
A'(s)c^{(k)}(t)=-\viertel 
\left\{
c^{(k+1)}(t-s)-c^{(k+1)}(t+s)
\right\}
\een
or rather
\ben
A'(0)c^{(k)}(t)=0\,.
\een
Since for some $t_0$ the vectors $\dot{c}(t_0),\ldots,c^{(n)}(t_0)$
are linearly independent we obtain $A'(0)=0$ 
and hence
by Equation~\eqref{eq:cs2}: $b'(0)=0.$
%%%%%%%%%%%%%%%%%%%%%%%%%%%%%%%%%%%%%%%%%%%%%%%%%%%%%%%%%%

We conclude from Equation~\eqref{eq:cs2}:
\begin{eqnarray*}
\frac{\partial^2 \cas (t)}{\partial s^2}&=&A''(s)c(t)+b''(s)\\
&=&
\frac{\partial^2 \cas (t)}{\partial t^2}-
\halb \, \ddot{c}(t)=
\left(\as-\Identity/2\right)\ddot{c}(t)\,.
\end{eqnarray*}
Since $A(0)=\Identity$
the endomorphisms $\as-\Identity/2$ are isomorphisms for
all $s \in (0,\epsilon)$ for a sufficiently small
$\epsilon>0\,.$ Hence we obtain for 
$s \in (0,\epsilon):$
\begin{equation}
\label{eq:odecs}
\ddot{c}(t)=\left(\as-\Identity/2\right)^{-1}A''(s)c(t)
+ 
\left(\as-\Identity/2\right)^{-1}b''(s)\,.
\end{equation}
Differentiating with respect to $s:$
\ben
%\label{eq:asminus1}
\left(
\asminus
A''(s)\right)'\left(c(t)\right)
+
\left(
\asminus
b''(s)\right)'
=0
\een
and differentiating with respect to $t:$
\ben
\left(
\asminus
A''(s)\right)'c^{(k)}(t)=0\,;\,k=0,1,2,\ldots,n\,.
\een
Since 
$\dot{c}(t_0),\ddot{c}(t_0),\ldots,c^{(n)}(t_0)$ 
are linearly independent by assumption we conclude
that 
$$\left(
\asminus
A''(s)\right)'=0\,.$$ 
Hence with $B=2A''(0)
, d=2 b''(0)$ 
we obtain 
\begin{equation*}
A''(s)=
\asa B
\,;\,b''(s)=\asa (d)
\,,
\end{equation*}
i.e.~the differential 
Equations~\eqref{eq:a2s}.
For $s=0$ we obtain from Equation~\eqref{eq:odecs}:
\ben
%\label{eq:odecc}
\ddot{c}(t)=B c(t)+d\,,
\een
i.e.~Equation~\eqref{eq:ddotc}.
\end{proof}
%%%%%%%%%%%%%%%%%%%%%%%%%%%%%%%%%%%%%%%%%%%%%%%%%%%%%%%%%
We can combine the results of Theorem~\ref{thm:characterization}
and Proposition~\ref{pro:odeb} to give the following 
characterization of affinely invariant curves under the
affine mapping $\Ta$ as solutions of an inhomogeneous 
linear differential equation of second order:
%%%%%%%%%%%%%%%%%%%%%%%%%%%%%%%%%%%%%%%%%%%%%%%%%%%%%%%%%%
\begin{theorem}[\sc characterization
of affinely invariant curves as solutions of an ode]

\label{thm:main}

(a) Let $c: \R \longrightarrow \R^n$ be a smooth curve
affinely invariant under the mapping $\Ta.$ 
Assume in addition that for some $t_0\in \R$ the 
vectors $\dot{c}(t_0),\ddot{c}(t_0),\ldots,c^{(n)}(t_0)$ 
are linearly independent. Then there is a unique 
matrix $B \in \mnr$ and a unique vector
$d \in \R^n$ such that $$\ddot{c}(t)=B c(t)+d.$$

\smallskip

(b) Let $B$ be a real matrix and $d$ a vector in $\R^n.$ Then 
any solution $c=c(t)$ of the inhomogenous linear differential
equation $\ddot{c}(t)=B c(t)+d$
with constant coefficients
defines an affinely invariant smooth curve
under the mapping $T.$
\end{theorem}
%%%%%%%%%%%%%%%%%%%%%%%%%%%%%%%%%%%%%%%%%%%%%%%%%%%%%%%%%%%%
%%%%%%%%%%%%%%%%%%%%%%%%%%%%%%%%%%%%%%%%%%%%%%%%%%%%%%%%
For a closed polygon $x \in \poN$ the {\em center of mass} $x_{cm}$
is given by
\begin{equation}
x_{cm}=\frac{1}{N}\sum_{j=1}^N x_j\,.
\end{equation}
Since 
$\left(T(x)\right)_{cm}=x_{cm}$ we conclude:
There is no translation-invariant closed polygon, 
since the center of mass is preserved under
the curve shortening processes $B$ 
or rather
$T.$
% % % % % % % % % % % % % % % % % % % % % %
% % % % % % % % % % % % % % % % % % % %
\begin{remark}[\sc Generalization of the map $T$]
\label{rem:talpha}

\rm
For three points $x,y,z \in \R^n$ define the
affine map 
$$\T: \R^n\times \R^n \times \R^n \longrightarrow \R^n\,;\,
\T(x,y,z)=\frac{1}{4}\left\{x+2y+z\right\}\,.$$
Hence the mapping 
$T: \po (\R^n)\longrightarrow\po (\R^n)$ 
introduced in Definition~\ref{def:selfsimilar}
satisfies for all $j \in \z:$
$$
T(x)_j=\T\left(x_{j-1},x_j,x_{j+1}\right)\,.
$$
The one-parameter family $c_s$ associated to
the smooth curve $c$ by 
Equation~\eqref{eq:oneparameter}
can be written as:
\ben
c_s(t):=\T\left(c(t-s),c(t),c(t+s)\right)=
\frac{1}{4}\left\{c(t-s)+2 c(t)+c(t+s)\right\}\,.
\een
In the following we will allow a slightly more general
curve shortening process 
\be
\label{eq:talpha}
T_{\alpha}: \po (\R^n)\longrightarrow
\po(\R^n),\left( T_{\alpha}(x)\right)_j:=\T_{\alpha}\left(
x_{j-1},x_j,x_{j+1}\right)
\ee
based on the affine
map
$\T_{\alpha}:\R^n \longrightarrow \R^n\,;
\T_{\alpha} (x,y,z)=\alpha x+(1-2\alpha)y+\alpha z$
for $\alpha \not=0,$ i.e.~$\T=\T_{1/4}.$
For $\alpha=1/3$ the point $\T_{1/3}(x,y,z)$ is the 
{\em center of mass}
$(x+y+z)/3.$
The curve $\alpha \in[0,1/2]\mapsto 
\T_{\alpha}\left(x_{j-1},x_j,x_{j+1}\right)
\in \R^n$
is a parametrization of the 
straight line connecting $x_j$ with the 
midpoint $(x_{j-1}+x_{j+1})/2$
of the points $x_{j-1},x_{j+1}.$
These mappings $T_{\alpha}$ are considered for example 
in~\cite[p.238-39]{Berlekamp1965} and
\cite[ch.5.1]{Bruckstein1997}.
For a smooth curve $c$ one defines the associated
one-parameter family of curves
\ben
c_{\al,s}(t)
=\T_{\alpha}\left(x(t-s),c(t),c(t+s)\right)=
\alpha c(t-s)+\left(1-2\alpha\right)c(t)
+\alpha c(t+s)\,.
\een
We call a smooth curve $c$ {\em affinely invariant}
(or a {\em soliton}) with
respect to $T_{\alpha}$ if there is a one-parameter
family $\left(A_{\alpha}(s),b_{\alpha}(s)\right), s
\in(-\epsilon,\epsilon)$ for some $\epsilon>0$
of affine mappings such that
$
c_{\alpha,s}(t)=A_{\alpha}(s)c(t)+b_{\alpha}(s)
$
for all $t\in \R, s\in (-\epsilon,\epsilon).$
Then
\begin{eqnarray*}
c_s(t)=c_{1/4,s}(t)=c(t)+\viertel\left\{c(t-s)-2c(t)+c(t+s)
\right\}\\
=\frac{(4\alpha-1)c(t)+c_{\alpha,s}(t)}{4\alpha}
\end{eqnarray*}
and
\ben
c_{\alpha,s}(t)=4\alpha c_s(t)-(4\alpha-1)c(t)\,.
\een
We conclude:
A smooth curve $c$ is a soliton
for the transformation $T_{\alpha}$ for $\alpha\not=0$ if
and only if it is a soliton for the transformation
$T=T_{1/4}$ and for the corresponding affine maps
$\left(A_{\alpha}(s),b_{\alpha}(s)\right) $
we obtain:
$
A_{\alpha}(s)=4\alpha \as-(4\alpha-1)\Identity\,;\,
b_{\alpha}(s)=4\alpha \bs\,.
$
\end{remark}
% % % % % % % % % % % % % % % % % % % % % % % % % % % % % %
% % % % % % % % % % % % % % % % % % % % % % % % % % % % % %
% % % % % % % % % % % % % % % % % % % % % % % % % % % % % % %
\section{Semidiscrete flows of polygons}
% % % % % % % % % % % % % % % % % % % % % % %
\label{sec:semidiscrete}
% % % % % % % % % % % % % % % % % % % % % % % %
The mapping $T$ introduced in 
Definition~\ref{def:selfsimilar} or $T_{\alpha}$ defined in
Remark~\ref{rem:talpha}
or Equation~\eqref{eq:talpha} can be seen as a
discrete version of the {\em semidiscrete flow}
defined on the space $\po(\R^n)$ of polygons:
For a given polygon $(x_j)_{j\in \z}$ the flow
$s \mapsto \left(x_j(s)\right)\in \po(\R^n)$ is defined by
the Equation
\be
\label{eq:semi}
\frac{d x_j(s)}{d s}=
x_{j-1}(s)-2 x_j(s)+ x_{j+1}(s)\,,\, x_j(0)=x_j\,.
\ee
This flow is discussed for example in~\cite{Chow2007}.
It is a linear first order  system of differential equations
with constant coefficients forming a circulant matrix. 
Hence one can write down the solutions explicitely.
If we approximate the left hand side of this
equation by
$\left(x_{j}(s+\alpha)-x_{j}(s)\right)/\alpha$
we obtain 
\begin{eqnarray*}
x_j(s+\alpha)
=\alpha x_{j-1}(s)+
\left(1-2\alpha\right)x_j(s)
+\alpha x_{j+1}(s)
\\
=
\T_{\alpha}\left(x_{j-1}(s), x_j(s),x_{j+1}(s)\right)\,.
\end{eqnarray*}
Therefore  the mappings $T$ and $T_{\alpha}$ can be
seen as discrete versions of the 
flow Equation~\eqref{eq:semi}.
In~\cite{Chow2007}
the flow $s \mapsto x_j(s)$ introduced in
Equation~\eqref{eq:semi} is called {\em semidiscrete}
since it is a smooth flow on a space of discrete
objects (polygons). Then a discretization of
the semidiscrete flow yields to the discrete
process $T$ and $T_{\alpha}$ discussed here.
On the other hand the connection of the semidiscrete
flow with the {\em smooth} curve shortening
flow in Euclidean space is discussed in detail in
\cite[Sec. 5]{Chow2007}.

%%%%%%%%%%%%%%%%%%%%%%%%%%%%%%%%%%%%%%%%%%%%%%%%%%%%%%
%%%%%%%%%%%%%%%%%%%%%%%%%%%%%%%%%%%%%%%%%%%%%%%%%%%%%%%
%%%3. Aenderung
%%%%%%%%%%%%%%%%%%%%%%%%%%%%%%%%%%%%%%%%%%%%%%%%%%%%%%%%
If we consider the  functional
$$F_2: \poN \longrightarrow \R\,;\,
F_2(x)=\halb \sum_{j=0}^N \left\|x_{j+1}-x_j\right\|^2$$
on the space $\poN$ of closed polygons
then we obtain for a curve $s
\in (-\epsilon,\epsilon) \mapsto 
x(s)=\left(x_j(s)\right)_{j\in \z_{N}}$
with $x=x(0)$ and $\dot{x}=\dot{x}(0):$
\ben
\left.\frac{d F_2(x(s))}{ds}\right|_{s=0}=
-\sum_{j=0}^N \left\langle \dot{x}_j,x_{j-1}-2x_j+x_{j+1}\right\rangle
\een
and we obtain for the gradient $\grad F_2(x):$
\ben
\grad F_2(x)=-\left(x_{j-1}-2x_j+x_{j+1}\right)_{j\in \z_N}\,.
\een
Hence the semidiscrete flow can be viewed as the 
negative gradient flow of
the functional $F_2,$ cf.~\cite[Sec.6]{Chow2007}.
An affine transformation $x\in \rn \mapsto \mathcal{A}(x)=A(x)+b \in \rn$
on $\rn$ induces an affine transformation 
$\widehat{\mathcal{A}}$
on
$\poN: \widehat{\mathcal{A}}=
\left(x_1,\ldots,x_n\right)=
\left(\mathcal{A}\left(x_1\right),\ldots, \mathcal{A}\left(x_N\right)\right).$
%%%%%%%%%%%%%%%%%%%%%%%%%%%%%%%%%%%%%%%%%%%%%%%%%%%%
%%%%%%%%%%%%%%%%%%%%%%%%%%%%%%%%%%%%%%%%%%%%%%%%%%%%
%%%%% 4. Aenderung
%%%%%%%%%%%%%%%%%%%%%%%%%%%%%%%%%%%%%%%%%%%%%%%%%%%%%
In contrast to the functional $F_2$ its gradient
$\grad F_2$ is invariant under
$\widehat{\mathcal{A}}: $
$$ \grad F_2\left( 
\widehat{\mathcal{A}}\left(x\right)\right)=
-\left(A\left(x_{j-1}\right)-2
A\left(x_j\right)+A\left(x_{j+1}\right)\right)_{j\in \z_N}=
\widehat{\mathcal{A}}\left(\grad F_2 \left(x\right)\right)\,.
$$

%%%%%%%%%%%%%%%%%%%%%%%%%%%%%%%%%%%%%%%%%%%%%%%%
%%%%%%%%%%%%%%%%%%%%%%%%%%%%%%%%%%%%%%%%%%%%%%%%
\begin{definition}[\sc affine invariance under the semidiscrete flow]

\label{def:semi}
We call a smooth curve $c:\R \rightarrow \R^n$
{\em affinely invariant} 
under the semidiscrete flow given by Equation~\eqref{eq:semi} 
(or a {\em soliton}) if there is some $\epsilon>0$ such that 
for any $s\in(0,\epsilon)$ there is an affine 
map $\left(\widetilde{A}(s),\widetilde{b}(s)\right)$ 
such that the one-parameter family
\be
\widetilde{c}_s(t)
=
\widetilde{A}(s)c(t)+\widetilde{b}(s)
\ee
is a solution of the flow equation
\ben
\frac{\partial \widetilde{c}_s(t)}{\partial s}=
\widetilde{c}_s(t-1)-2 \widetilde{c}_s(t)+
\widetilde{c}_s(t+1)
\een
for all $s\in (0,\epsilon),t\in \R\,.$
%resp. the heat Equation~\eqref{eq:heat}.
\end{definition}
%%%%%%%%%%%%%%%%%%%%%%%%%%%%%%%%%%%%%%%%%%%%%%%%%%%%%%%
%%%%%%%%%%%%%%%%%%%%%%%%%%%%%%%%%%%%%%%%%%%%%%%%%%%%%%%
In the following Proposition we
show that the solitons of the
mapping $T$ or $T_{\alpha}$ coincide with
the solitons of the semidiscrete flow given by Equation~\eqref{eq:semi}:
%%%%%%%%%%%%%%%%%%%%%%%%%%%%%%%%%%%%%%%%%%%%%%%%%%%%%%%%
%%%%%%%%%%%%%%%%%%%%%%%%%%%%%%%%%%%%%%%%%%%%%%%%%%%%%%%%
\begin{proposition}[\sc affinely 
invariant curves under the semidiscrete flow]
\label{pro:semidiscrete}
Let $B$ be a matrix and $d$ a vector in $\R^n$ and
let 
$c=c(t)$ be a solution of the inhomogenous linear differential
equation $\ddot{c}(t)=B c(t)+d$
with constant coefficients 
for which 
$\dot{c}(0),\ddot{c}(0),\ldots,c^{(n)}(0)$
are linearly independent.
Then the curve $c$
defines an affinely invariant smooth curve
under the semidiscrete flow
given in Equation~\eqref{eq:semi}\,.
\end{proposition}
%%%%%%%%%%%%%%%%%%%%%%%%%%%%%%%%%%%%%%%%%%%%%%%%%%%
%%%%%%%%%%%%%%%%%%%%%%%%%%%%%%%%%%%%%%%%%%%%%%%%%%%%
%%%%%%%%%%%%%%%%%%%%%%%%%%%%%%%%%%%%%%%%%%%%%%%%%%%%
\begin{proof}
We conclude from 
Proposition~\ref{pro:odeb}
and Equation~\eqref{eq:csA} that there is 
a matrix $A_1=(A(1)-\Identity)/4$ and a vector $b_1=b(1)/4$
such that
\be
\label{eq:cteins}
c(t-1)-2c(t)+c(t+1)=A_1 c(t)+b_1
\ee
for all $t \in \R.$
Let
\ben
\widetilde{A}(s)=\exp(A_1\cdot s)\,;\,
\widetilde{b}(s)=\int_0^s \exp(A_1\cdot \sigma)(b_1)\,d\sigma\,,
\een
and
\be
\label{eq:cstw}
\widetilde{c}_s(t)= \widetilde{A}(s) c(t)+\widetilde{b}(s)\,.
\ee
Using $A_1 \cdot \widetilde{A}(s)=
\widetilde{A}(s) \cdot A_1$ we obtain:
\begin{eqnarray*}
\frac{\partial \widetilde{c}_s(t)}{\partial s} 
=\widetilde{A}'(s) c(t)+\widetilde{b}'(s)=
A_1 \cdot \widetilde{A}(s) c(t)+
\widetilde{A}(s)(b_1)\\
=\widetilde{A}(s)\left(A_1c(t)+b_1\right)\,.
\end{eqnarray*}
We conclude from Equation~\eqref{eq:cteins}
and Equation~\eqref{eq:cstw}:
\ben
\frac{\partial \widetilde{c}_s(t)}{\partial s} 
=\widetilde{A}(s)\left(
c(t-1)-2c(t)+c(t+1)
\right)=
\widetilde{c}_s(t-1)-2 \widetilde{c}_s(t)+
\widetilde{c}_s(t+1)\,.
\een
Hence the curve $c$ is affinely invariant
under the semidiscrete flow.
\end{proof}
%%%%%%%%%%%%%%%%%%%%%%%%%%%%%%%%%%%%%%%%%%%%%%%%%%%%%%%%
%%%%%%%%%%%%%%%%%%%%%%%%%%%%%%%%%%%%%%%%%%%%%%%%%%%%%%%%
%%%%%%%%%%%%%%%%%%%%%%%%%%%%%%%%%%%%%%%%%%%%%%%%%%%%%%%%
\section{Planar solitons }
%%%%%%%%%%%%%%%%%%%%%%%%%%%%%%%%%%%%%%%%%%%%%%%%%%%%%%%%
\label{sec:planar}
%%%%%%%%%%%%%%%%%%%%%%%%%%%%%%%%%%%%%%%%%%%%%%%%%%%%%%%%%
%%%%%%%%%%%%%%%%%%%%%%%%%%%%%%%%%%%%%%%%%%%%%%%%%%%%%%%%
\begin{figure}[t]
  \centering
  \begin{minipage}[b]{6 cm}
      \includegraphics[width=0.85\textwidth]{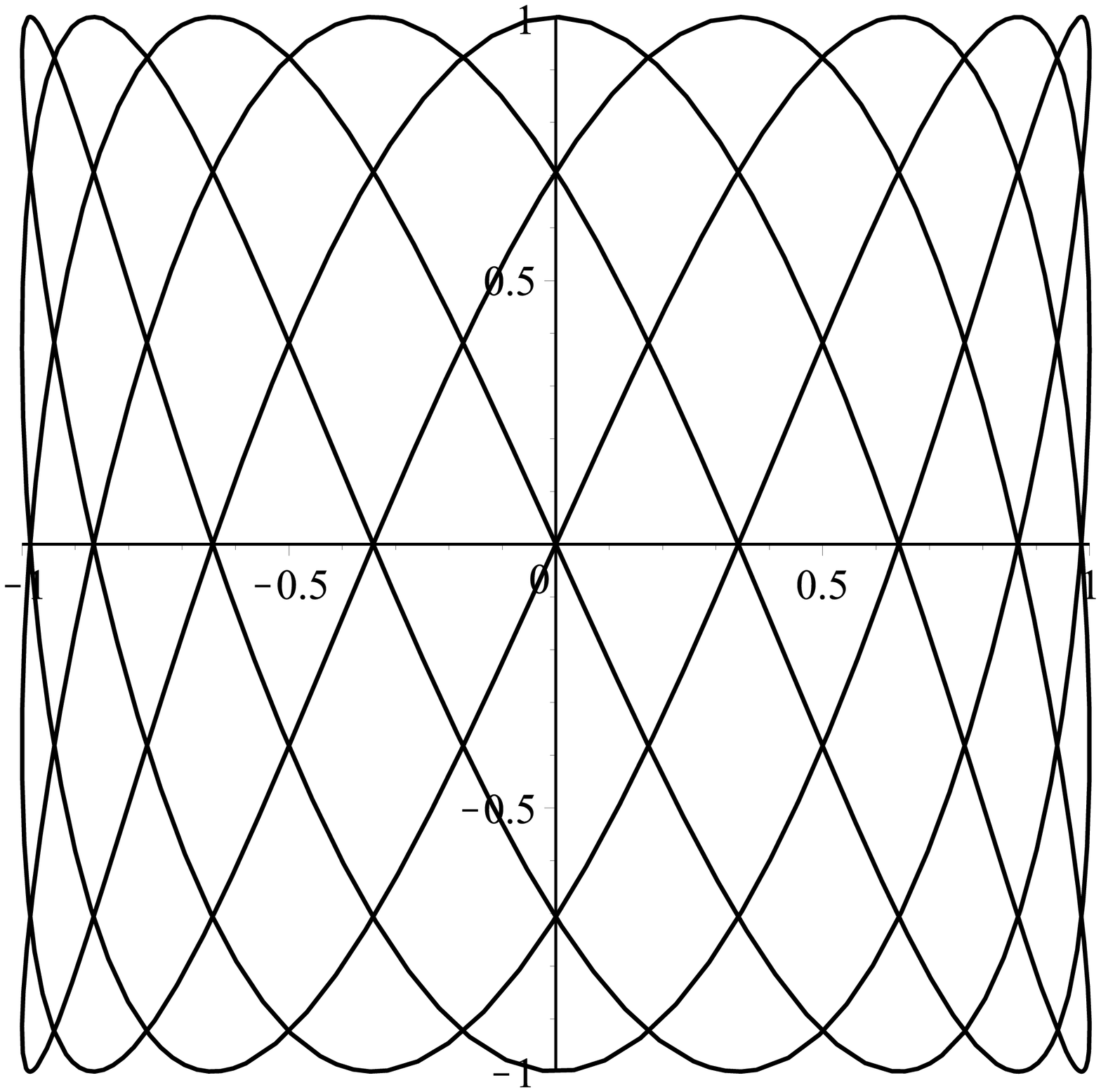} 
      \caption{Scaling, Case (1a)}
      \label{fig:lissajous}
    \end{minipage}
     \begin{minipage}[b]{6 cm}
          \includegraphics[width=0.85\textwidth]{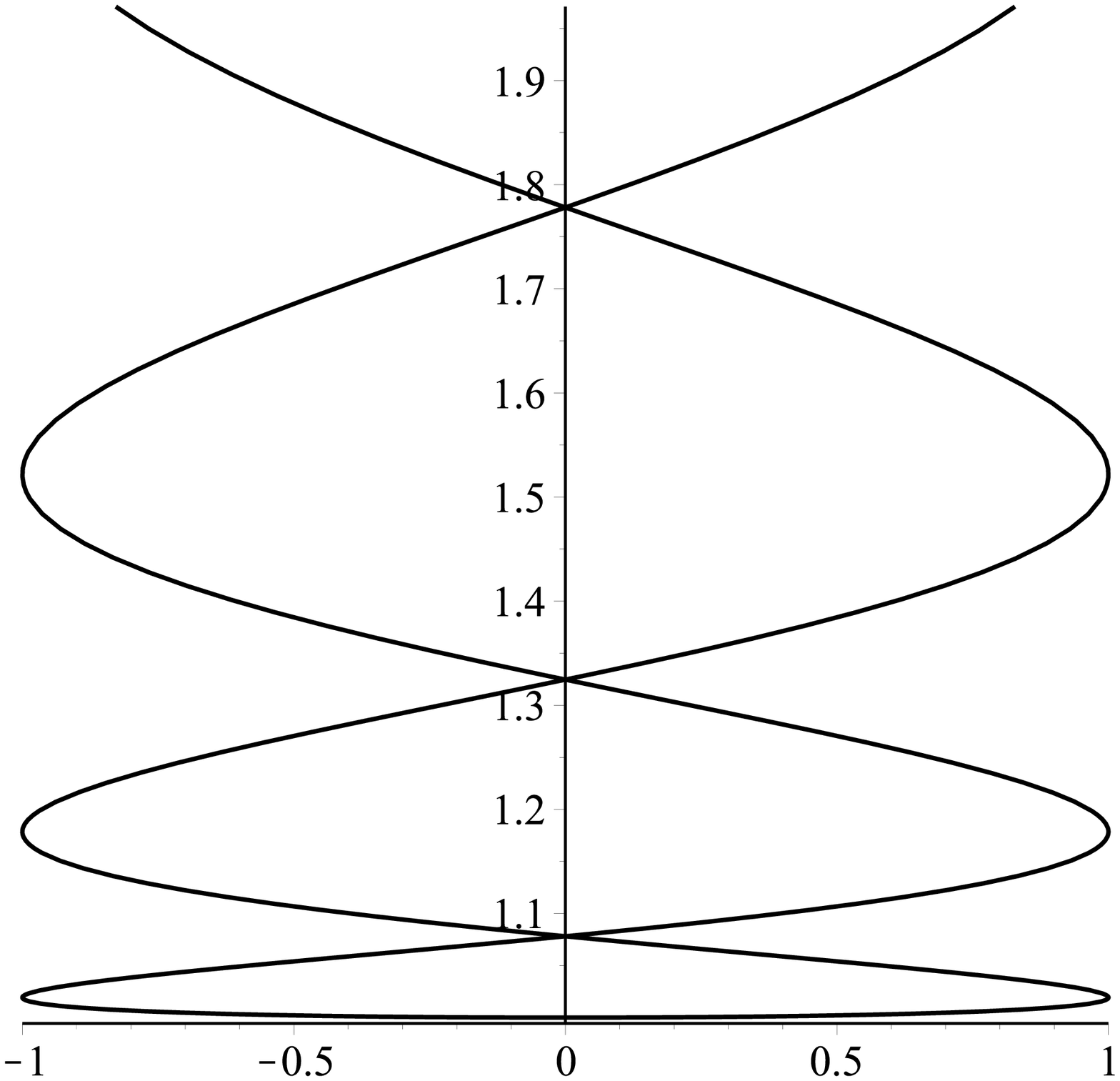}  
          \caption{Scaling, Case (1b)}
          \label{fig:spiral6}
        \end{minipage}
           \end{figure}
% % % % % % % % % % % % % % % % % % % % % % % %
% % % % % % % % % % % % % % % % % % % % % % % %
% % % % % % % % % % % % % % % % % % % % % % % %

% % % % % % % % % % % % % % % % % % % % % % % % % %
% % % % % % % % % % % % % % % % % % % % % % % % % % %
% % % % % % % % % % % % % % % % % % % % % % % % %
We study the planar case $n=2.$
We conclude from Theorem~\ref{thm:main} that 
solitons $c$ are 
solutions $c(t)=\left(x(t),y(t)\right)$
of the differential equation
$$\ddot{c}(t)=B c(t)+d\,.$$
We discuss these solutions using 
Proposition~\ref{pro:ode} and
Proposition~\ref{pro:odeb}.

Let $B\in \mtwor$ be a matrix in
(real) Jordan normal form. Then
we consider the following
cases:

\begin{enumerate}[leftmargin=0.6cm]
\item
$B$ is diagonalizable (over $\R$)
and invertible, and $d=0.$ 
i.e.
\ben
B=
\left(
\begin{array}{cc}
b_1 &0\\
0&b_2
\end{array}
\right)\,,\, b_1,b_2\in \R-\{0\}\,,
\een
this affine map is called {\em scaling.}
If the diagonal entries coincide the transformation is also
called a {\em homothety.}
The differential equation $\ddot{c}=B c$
implies $\ddot{x}=b_1x, \ddot{y}=b_2y.$
With the notation introduced in Equation~\eqref{eq:coa}
we obtain with $c(0)=(v_1,v_2), \dot{c}(0)=(w_1,w_2):$
\ben
c(t)=\left(v_1\cos_{b_1}(t)+w_1\sin_{b_1}(t),
v_2\cos_{b_2}(t)+w_2\sin_{b_2}(t)\right)
\een
The matrices $\as$ are also diagonal matrices (or scalings)
\ben
\as= \halb \left\{\Identity+\coA(s)\right\}=
\halb
\left(
\begin{array}{cc}
1+\cos_{b_1}(s) &0\\
0&1+\cos_{b_2}(s)
\end{array}
\right)
.
\een
In particular the diagonal entries of the matrices
$\as$ are non-negative for all $s.$
Hence the one parameter family $c_s$ of curves
is produced by {\em scaling} (with a diagonal
matrix) from the soliton $c=c_0.$

Particular examples are:

\smallskip

\begin{enumerate}[leftmargin=0.4cm]
\item[(1a)] 
We obtain {\em closed curves} if $b_1=-\lambda_1^2<0,b_2
=-\lambda_2^2<0$ and 
$\lambda_1/\lambda_2 \in \mathbb{Q}.$
Then we obtain for example {\em Lissajous}
curves of the form
\ben
c(t)=\left(w_1 \sin \left(\lambda_1 t\right),
v_2 \cos\left( \lambda_2 t\right)\right)
\een
see Figure~\ref{fig:lissajous} for
$\lambda_1=4,\lambda_2=9,w_1=v_2=1, w_2=v_1=0$
and $0\le t\le 6.3.$
Another example is given in the Introduction,
it is the curve
$
c(t)=\left( \cos \left(2 t\right),
 \cos\left( 3 t\right)\right)\,.
$
In Figure~\ref{fig:polygon2} we show 
the curve $c=c(t)$ and $c_{0.4}=c_{0.4}(t)=A\left(0.4\right) c(t):$ 
\ben
c_{0.4}(t)=
\halb
\left(\left(1+\cos\left(0.8\right)\right)
\cos(2t),
\left(1+\cos\left(1.2\right)\right)
\cos(3t)
\right)\,.
\een

\item[(1b)]
If $b_1=\lambda_1^2>0, b_2=-\lambda_2^2<0$
we obtain for example
curves of the following form:
\ben
c(t)=\left(w_1\sin\left(\lambda_1 t\right),
v_2\cosh\left(\lambda_2 t\right)\right)
\een
In Figure~\ref{fig:spiral6} this curve is shown
for the parameters $
\lambda_1=8, \lambda_2=1, v_2=w_1=1, v_1=w_2=0$
and $-1.3\le t\le 1.3.$

              \begin{figure}[t]
              \centering
\begin{minipage}[b]{6 cm}
\includegraphics[width=0.9\textwidth]{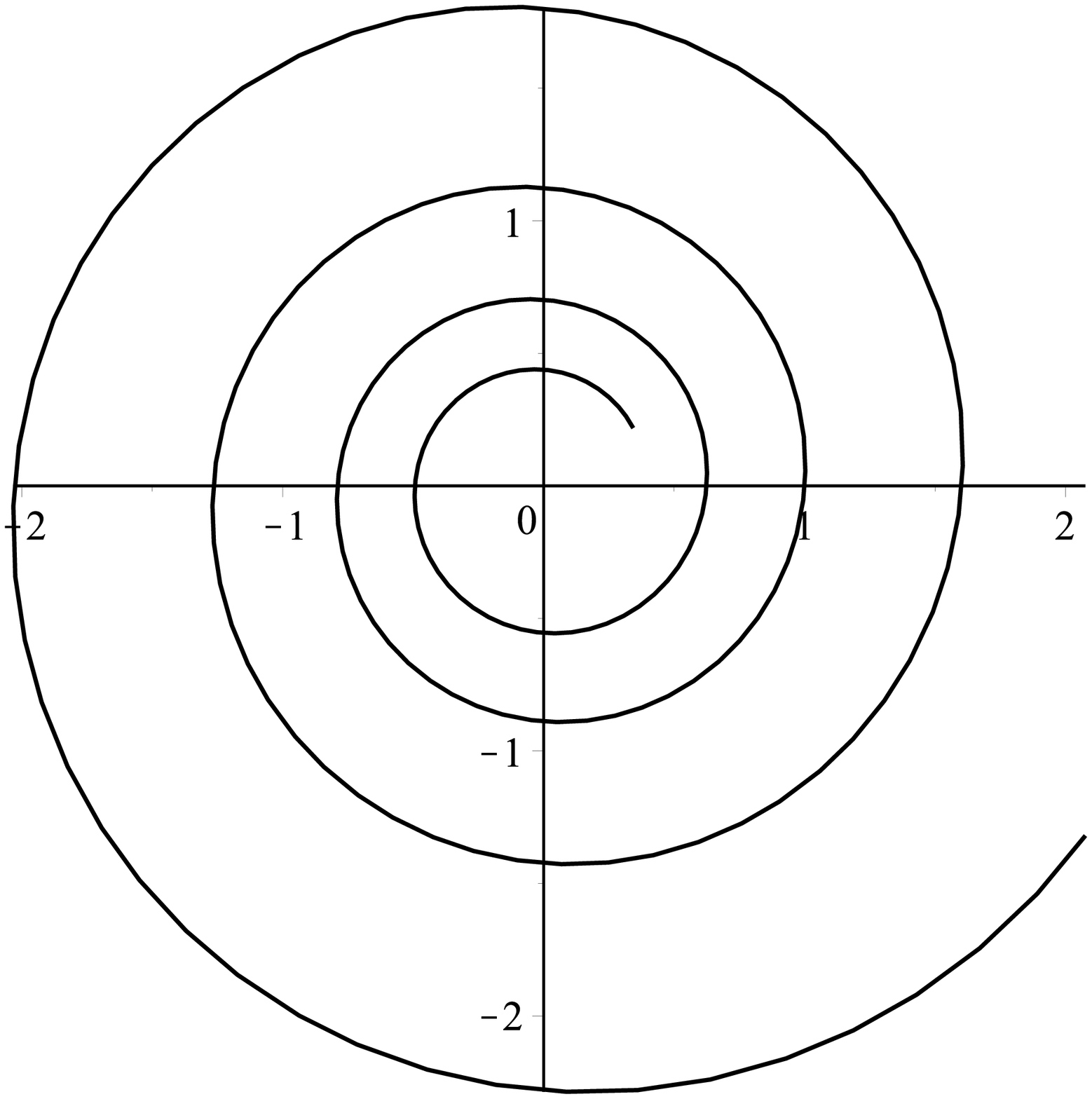}  
\caption{{\em Spira mirabilis,} 
Rotate and scale, Case (2a)}
\label{fig:mirabilis}
\end{minipage}
    \begin{minipage}[b]{6 cm}
    \includegraphics[width=0.9\textwidth]{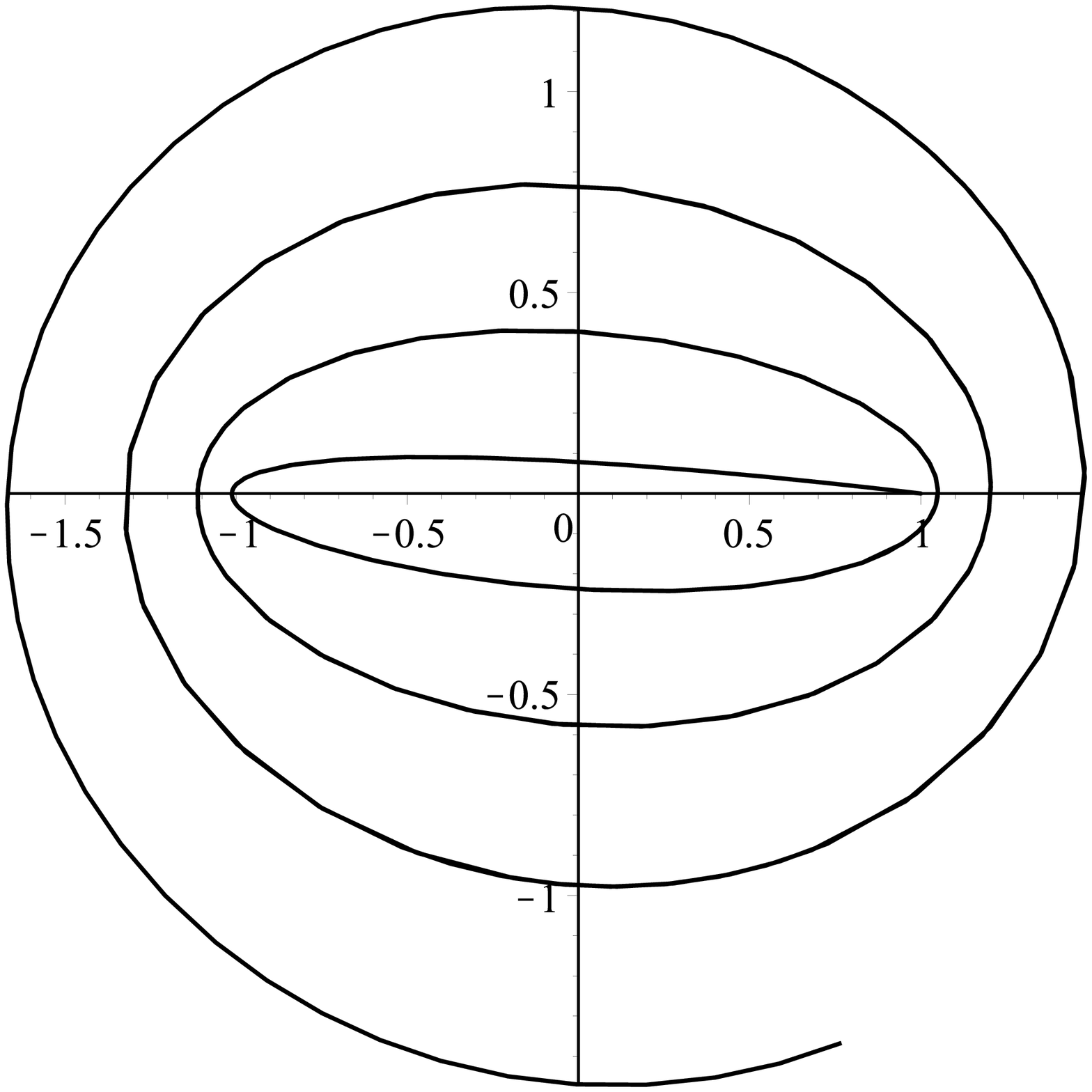} 
    \caption{Rotate and scale, Case (2b)}
    \label{fig:spiral0}
  \end{minipage}
  \end{figure}
  % % % % % % % % % % % % % % % % % % % %
\item[(1c)]
If $B$ is diagonalizable and has the form
\ben
A=
\left(
\begin{array}{cc}
0 &0\\
0&b_2
\end{array}
\right)\,,\,b_2 \not=0.
\een
and $d=0$
then $\ddot{x}(t)=0,\ddot{y}=b_2y,$ i.e.
$
c(t)=\left(t,v_2\cos_{b_2}(t)+w_2\sin_{b_2}(t)
\right)
$
are solutions, for example
\ben
c(t)=(t,\exp(t)), c(t)=(t,\cosh(t)), c(t)=(t,\sin(t))\,.
\een
\end{enumerate}
\item
The following case corresponds to the case that the
matrix $B$ has no real eigenvalues. Hence
$B$ is a {\em similarity,} i.e.~a composition of a {\em rotation}
and a {\em dilation} $x \mapsto \lambda x$ for some 
$\lambda \not=0.$
We identify $\R^2$ with the complex numbers $\complexnumbers$
and assume that the matrix $A$ is complex linear, i.e.~can be
identified with the multiplication with a 
non-zero complex number
$\mu$ 
Then we are looking for a solution $z:t \in\R
\mapsto z(t)=x(t)+i y(t)\in\complexnumbers$
of the differential equation
$\ddot{z}=\mu z.$ For a complex number $w$ 
with $\mu=w^2, \; w=u_1+i u_2, u_1,u_2\in \R,
u_1=\Re(w),u_2=\Im(w)$
a solution has the form
$$
z(t)= h_1\exp(wt)+h_2\exp(-wt).
$$
for $h_1,h_2\in \complexnumbers.$
If we write $h_1=h_{11}+ih_{12}, h_2=h_{21}+ih_{22}$
with $h_{11},h_{12},h_{21},h_{22} \in \R$ then
we obtain for $x(t)=\Re(z(t)), y(t)=\Im(z(t)):$
\begin{eqnarray}
x(t)&=&\left\{h_{11} \exp(u_1t)+h_{21}\exp(-u_1t)\right\}\cos(u_2t)
\nonumber
\\
&&+ \left\{-h_{12} \exp(u_1t)+h_{22}\exp(-u_1t)\right\}\sin(u_2t)
\nonumber
\\
y(t)&=&
\left\{h_{12} \exp(u_1t)+h_{22}\exp(-u_1t)\right\}\cos(u_2t)
\nonumber
\\
\label{eq:xyt}
&&+\left\{h_{11}\exp(u_1t)-h_{21}\exp(-u_1t)\right\} \sin(u_2t)
\end{eqnarray}
%For $h_2=0$ it is the logarithmic spiral ({\em spira %mirabilis}).
For arbitrary $h_1,h_2$
the one-parameter family $z_s(t)$ 
in complex notation is given by
\begin{eqnarray*}
z_s(t)&=&\viertel \left\{z(t-s)+2z(t)+z(t+s)\right\}\\
&=&
\halb\left(1+\cosh(s w)\right)
\left(h_1 \exp(w t)+h_2 \exp(-w t)\right)\,.
\end{eqnarray*}
In real notation we obtain the matrix:
\begin{equation*}
\as=\halb
\left(\begin{array}{cc}
1+\cosh(u_1s)\cos(u_2s) & - \sinh(u_1s)\sin(u_2s)\\
\sinh(u_1s)\sin(u_2s) & 1+\cosh(u_1s)\cos(u_2 s)\,.
\end{array}
\right)
\end{equation*}
  \begin{figure}[t]
  \centering
  \begin{minipage}[b]{6 cm}
    \includegraphics[width=0.9\textwidth]{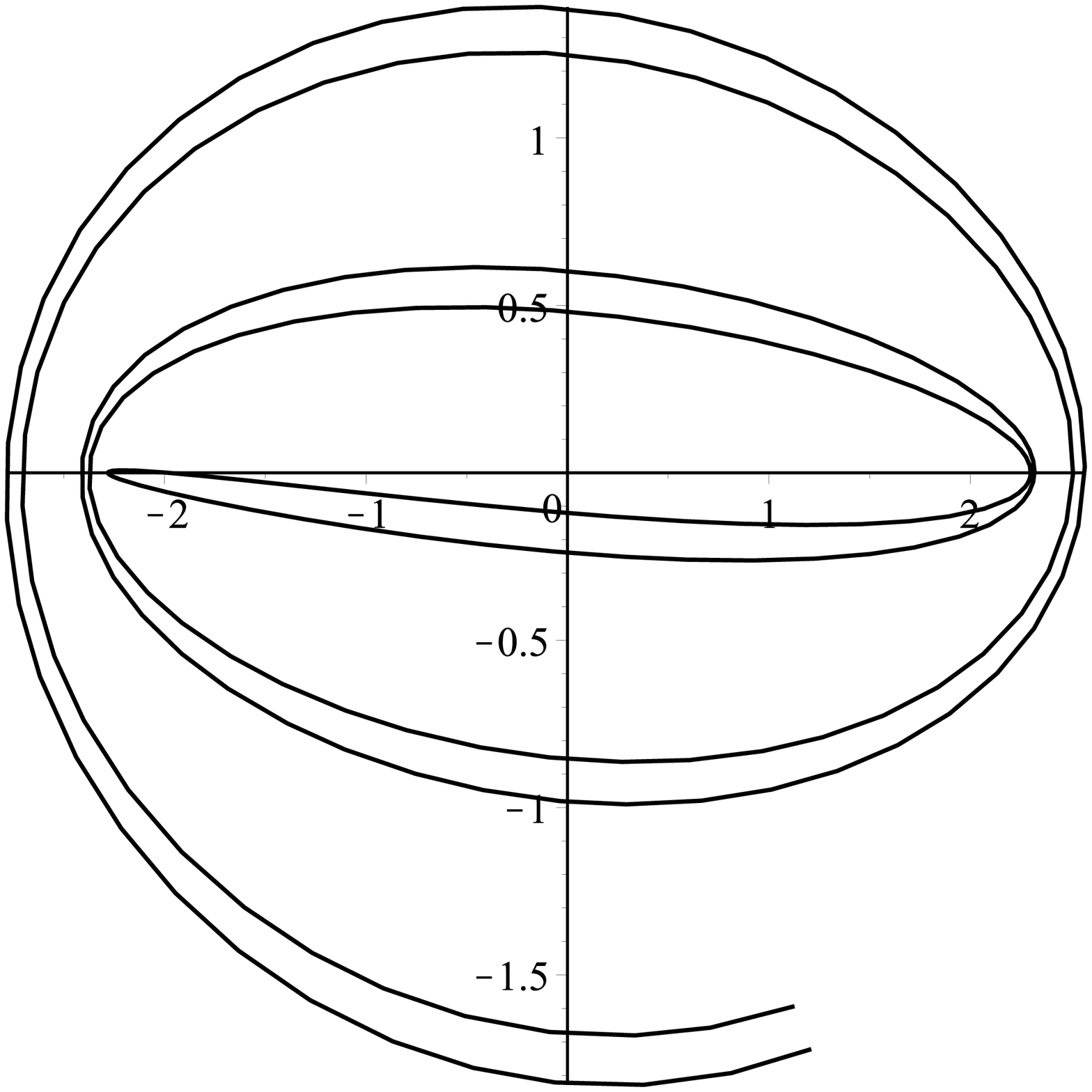}  
    \caption{Rotate and scale, Case (2c)}
    \label{fig:spiral1}
  \end{minipage}
  \begin{minipage}[b]{6 cm}
      \includegraphics[width=1.0\textwidth]{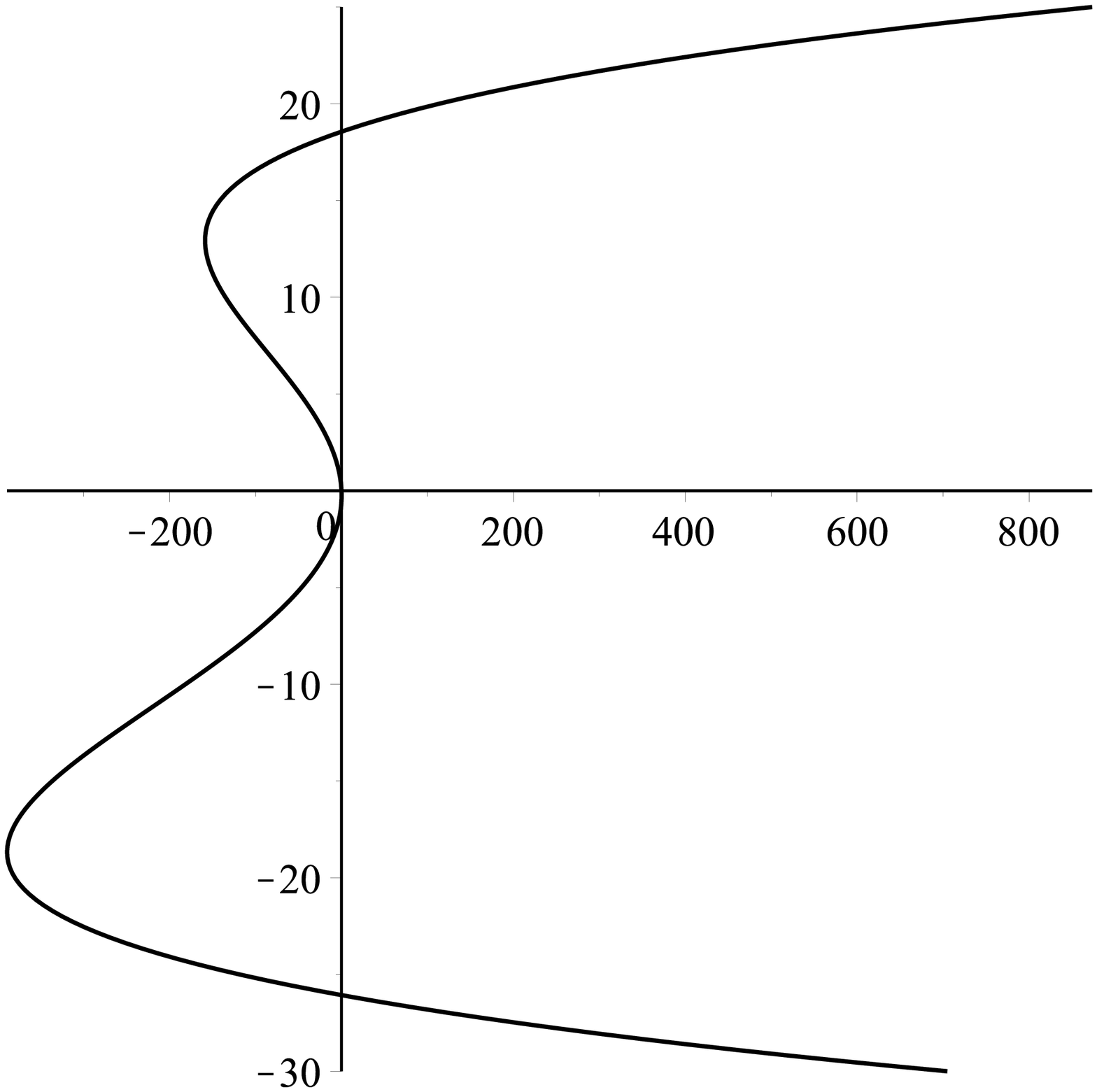} 
      \caption{Shear, Case (3)}
      \label{fig:seven}
    \end{minipage}
\end{figure}

Particular examples are:

\smallskip

\begin{enumerate}[leftmargin=0.5cm]
\item[(2a)] The logarithmic spiral ({\em spira mirabilis}): \ben
c(t)=\exp(u_1t)\left(\cos(u_2t),\sin(u_2t)\right)\,.
\een
i.e.~$h_{11}=1,h_{12}=h_{21}=h_{22}=0.$
Figure~\ref{fig:mirabilis} shows this curve 
for $u_1=0.3, u_2=4$ and $-3\le t \le 3.$
% % % % % % % % % % % % % % % % % % % % % %
\item[(2b)]
The curve
\ben
\label{eq:spiral0}
c(t)=\left( \cosh(u_1t) \cos(u_2t),
\sinh (u_1t) \sin(u_2t)
\right)
\een
(i.e.~$h_{11}=h_{21}=1/2, h_{12}=h_{22}=0$)
is shown in Figure~\ref{fig:spiral0}
for the values $u_1=1,u_2=20$
and $0 \le t \le 1.2.$
% % % % % % % % % % % % % % % % % % % % % % % % % %
\item[(2c)]
The curve given in Equation~\eqref{eq:xyt} for
$h_{11}=1,h_{21}=1.3, h_{12}=h_{22}=0, u_1=1, u_2=20$
is shown in 
Figure~\ref{fig:spiral1}
for $-0.57 \le t \le 0.885$.
\end{enumerate}

\item
Let $B$ be non-zero and nilpotent, 
i.e.
$$ B=
\left(
\begin{array}{cc}
0 &1\\
0& 0
\end{array}
\right)
$$ 
and $d=(0,d_2).$

Then
a solution is given by 
$
c(t)=
\left(d_2t^4/24+a_3 t^3/6+a_2t^2/2+a_1t+a_0, t\right)
$
for $a_0,a_1,a_2,a_3
\in \R.$
The matrices $\as$ are of the form:
\begin{equation*}
\as=
\left(
\begin{array}{cc}
1 & s^2/2\\
0 & 1
\end{array}
\right),
\end{equation*}
cf. the Proof of Proposition~\ref{pro:coA}.%Remark~\ref{rem:nilpotent}.
Hence the one-parameter family 
$s \mapsto c_s$ is formed by
{\em shear transformations.}
The curve with parameters
$d_2=0.1,a_3=0.2,a_2=-4, a_1=-1, a_0=0$ 
for $-30\le t\le 25$ is 
shown in Figure~\ref{fig:seven}.

% % % % % % % % % % % % % % % % % % % %

\item
Let $B$ be invertible with real
eigenvalue 
and
not diagonizable, i.e.
$$ B=
\left(
\begin{array}{cc}
b &1\\
0& b
\end{array}
\right)
$$
with $b \in \R, b\not=0$ and $d=0.$
Then $c(t)=\left(x(t),y(t)\right)$ with
\begin{eqnarray*}
x(t)&=&
\left(v_1+\frac{w_2}{2b}t\right)\coanew(t)+
\left(\frac{w_1}{b}-\frac{w_2}{2b^2}+
\frac{v_2}{2}t\right)\sianew(t)\\
y(t)&=&v_2\coanew(t)+w_2\sianew(t)
\end{eqnarray*}
The matrices $\as$ are of the form
\ben
\as=
\halb
\left(
\begin{array}{cc}
1+\coanew(s)&s \cdot \sianew(s)/2\\
0&1+\coanew(s)
\end{array}
\right)\,,
\een
cf. the Proof of Proposition~\ref{pro:coA}.
Hence the one-parameter family 
$s \mapsto c_s$ is formed by 
a composition of 
{\em shear transformations} and {\em scalings.}
For the parameters $b=-1, v_1=1,v_2=-0.1,w_1=-10,w_2=1$
and $-30\le t\le 40$ the curve is shown in
Figure~\ref{fig:shear-one}.
% % % % % % % % % % % % % % % % % % % % % % %

\begin{figure}[t]
  \centering
    \begin{minipage}[b]{6 cm}
    \includegraphics[width=0.9\textwidth]{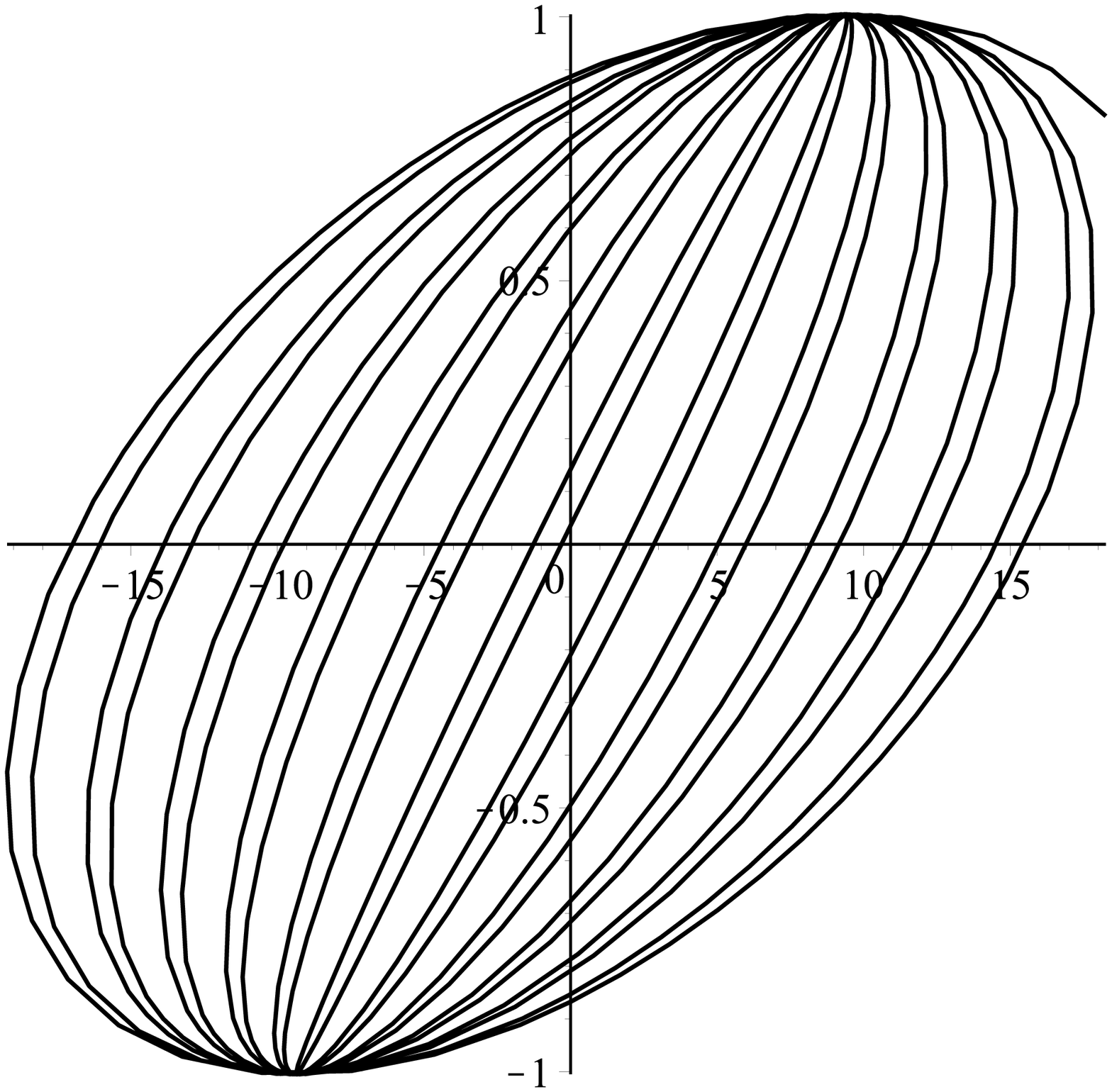} 
    \caption{Shear and Scaling, Case (4)}
    \label{fig:shear-one}
  \end{minipage}
  \begin{minipage}[b]{6 cm}
      \includegraphics[width=0.9\textwidth]{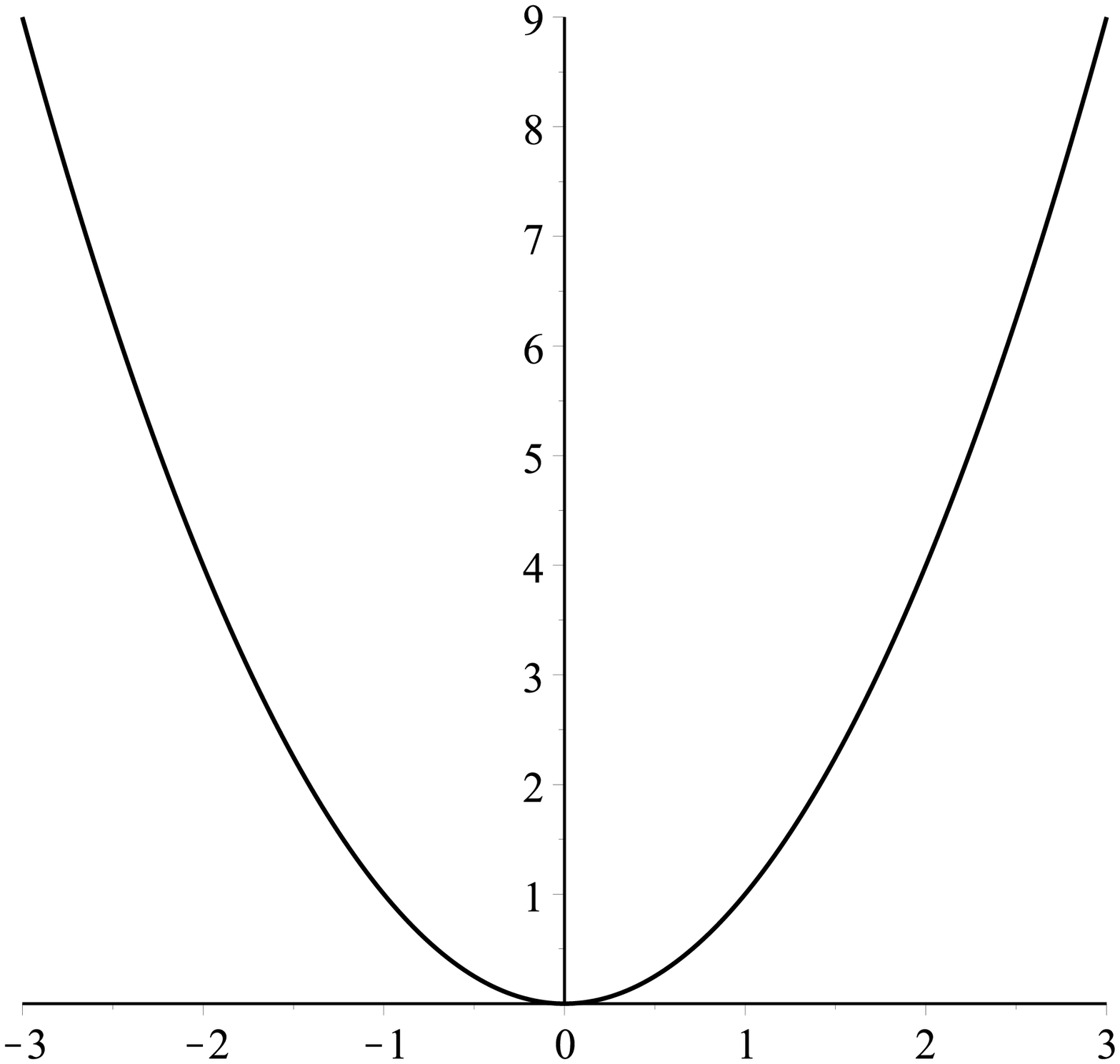}
%      \textwidth]{figure10} 
\caption{Translation, Case (5)}
%      \caption{Translation and Scaling, Case (6)}
     \label{fig:parabola}
    \end{minipage} 
    \end{figure}

% % % % % % % % % % % % % % % % % % % % % % %
\item 
If $B=0$ and $d \not=0,$ then
we obtain (up to an affine
transformations) the parabola 
$c(t)=(t,t^2)$ as translation-invariant curve,
cf. Proposition~\ref{pro:odeb}(c) and 
Figure~\ref{fig:parabola}.
% % % % % % % % % % % % % % % % % % % % % % % % % %
\begin{figure} %[t]
\includegraphics[width=0.4\textwidth]{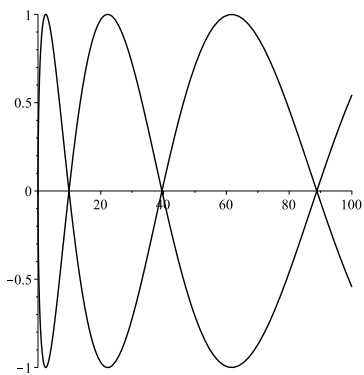} 
\caption{Translation and Scaling, Case (6)}
\label{fig:shear-two}
\end{figure}
% % % % % % % % % % % % % % % % % % % % % % % % % 
\item
If $B$ is of the form
$$ B=
\left(
\begin{array}{cc}
0 &0\\
0& b
\end{array}
\right)
$$
with non-zero $b$ and $b=(2,0).$
Then $\ddot{x}_1(t)=2$ and 
$\ddot{x}_2(t)=b x_2(t).$
Then
\ben
c(t)=\left(t^2,v_1\cos_{b}(t)+
w_1 \sin_{b}(t)\right)
\een
Examples are (for $b=\pm \lambda^2:$)
\begin{eqnarray*}
c(t)=\left(t^2, \cos(\lambda t)\right),
c(t)=\left(t^2,\exp(\lambda t)\right),
c(t)=\left(t^2,\cosh(\lambda t)\right)\,.
\end{eqnarray*}
In Figure~\ref{fig:shear-two} the curve
$c(t)=(t^2,\sin t), -10\le t\le 10$ is shown.
\end{enumerate}
%%%%%%%%%%%%%%%%%%%%%%%%%%%%%%%%%%%%%%%%%%%%%%%%%%%%%%%%%%%%%%%
%%%%%%%%%%%%%%%%%%%%%%%%%%%%%%%%%%%%%%%%%%%%%%%%%%%%%%%%%%%%%%%%%%%%%%%%%%%%%%%%%%%%%%%%%%%
%\end{itemize}
\bibliography{Discrete}
\bibliographystyle{plain}
%%%%%%%%%%%%%%%%%%%%%%%%%%%%%%%%%%%%%%%%%%%%%%%%%%%%%%%%
%%%%%%%%%%%%%%%%%%%%%%%%%%%%%%%%%%%%%%%%%%%%%%%%%%%%%%%%
\end{document}